\documentclass[10pt]{amsart}

\usepackage{epic}
\usepackage{eepic}
\usepackage{graphicx, color}
\usepackage{psfrag}

\graphicspath{{figures/}}

\newtheorem{thm}{Theorem}[section]
\newtheorem{ass}[thm]{Assumption}
\newtheorem{coro}[thm]{Corollary}
\newtheorem{lem}[thm]{Lemma}
\newtheorem{prop}[thm]{Proposition}

\theoremstyle{definition}

\newtheorem{defn}[thm]{Definition}

\newtheorem{nota}[thm]{Notation}

\theoremstyle{remark}

\newtheorem{remk}[thm]{Remark}

\newcommand{\supp}{{\rm supp}}

\newcommand{\ep}{\epsilon}

\newcommand{\R}{{\mathbb{R}}}

\newcommand{\gotq}{{\mathfrak q}}
\newcommand{\gotn}{{\mathfrak n}}
\newcommand{\gotN}{{\mathfrak n}}
\newcommand{\refeq}[1]{(\ref{#1})}

\newcommand{\Poly}{\mathcal{P}}

\begin{document}

\title[domains with small slits]{
      The eigenvalues of the Laplacian on domains with small slits}

\author{Luc Hillairet and Chris Judge}     

\date{\today}

\begin{abstract}
We introduce a small slit into a planar domain 
and study the resulting effect upon the eigenvalues of the 
Laplacian. In particular, we show that as the length
of the slit tends to zero, each real-analytic eigenvalue
branch tends to an eigenvalue of the original domain. 
By combining this with our earlier work \cite{HJ1},
we obtain the following application: The generic 
multiply connected polygon has simple spectrum.
\end{abstract}

\maketitle


\section{Introduction}

In this paper we study the following singular perturbation 
problem.  Let $\Omega \subset {\mathbb R^2}$ be a bounded open set 
having Lipschitz boundary. Remove from $\Omega$ the horizontal
slit, $\Sigma_t= \{ (x,y) |~ y=y_0,  |x-x_0| \leq t\}$,   
centered at the point $(x_0,y_0) \in \Omega$. Standard analytic 
perturbation theory shows that the eigenvalues of the Laplacian
on the slit domain $\Omega \setminus \Sigma_t$ depend
analytically on $t$ for $t>0$. (See Theorem \ref{Analytic}.) 
The main result of our study is the following. 

\begin{thm} \label{Main}
Let  $t \mapsto E_t$ be a real-analytic eigenvalue branch of the 
Laplacian acting on $L^2(\Omega \setminus \Sigma_t)$ 
with either Dirichlet or Neumann boundary conditions. Then 
$E_t$ converges to an eigenvalue of the Laplacian 
acting on $L^2(\Omega)$ as $t$ tends to $0$.
\end{thm}

We remark that the convergence of real-analytic eigenbranches is much
subtler than the convergence of ordered eigenvalues. For example,  
as $t$ tends to zero, the $k^{\rm{th}}$ eigenvalue of the 
Laplacian on the rectangle $[0,t] \times [0, 1/t]$ tends to zero,
but infinitely many analytic eigenvalue branches limit
to infinity.  Note that our proof  that each 
eigenbranch in Theorem \ref{Main} has a finite limit 
is more involved than the proof that ordered eigenvalues 
limit to an eigenvalue of $\Omega$.  
(Compare \S \ref{SecOrdered} and \S \ref{SecLimits}.) 

This work began as part of a study of the Laplace spectrum
of degenerating translation surfaces. Indeed, the 
collapse of a slit is one of the typical degenerations of 
a translation surface \cite{EMZ}. Theorem \ref{Main} deals with 
the simplest case of such a degeneration. The slit degeneration 
of a translation surface is analogous to the well-studied `neck 
pinching' degeneration of hyperbolic surfaces and our overall 
strategy for proving Theorem \ref{Main} mirrors 
the strategy employed in \cite{Wol} and \cite{Jud}.

The strategy is to control the negative variation of the 
logarithmic derivative, $\partial_t \log(E_t)$, for small $t$. 
A well-known formula---Proposition \ref{Variation}---relates $\partial_t E_t$
to certain quadratic forms defined as integrals over the domain. 
Not unexpectedly, one need only consider the contribution to these 
integrals that comes from a small neighborhood of the slit
(Corollary \ref{GlobalReduction}).  A judicious choice of coordinates 
in a neighborhood of the slit allows one to reduce the analysis of the
quadratic forms to a family of 1-dimensional problems indexed
by $i = 0, 1, 2, 3 \ldots$.  A simple convexity estimate
(Lemma \ref{SuperAdditive}) provides for control 
of the 1-dimensional contributions for $i >  (2t)^2 \cdot E_t$
where $2t$ is the width of the slit. We first show that 
$t^2 \cdot E_t$ is bounded and then use convexity for $i$ large
and compactness for $i$ small to find that  $t^{2k} \cdot E_t$
tends to zero for some $k<1$. Thus convexity applies to all 
$i> 0$, and a special estimate can be made for the case $i=0$. 

Elliptical coordinates in a neighborhood of the slit are
particularly well-suited for our purpose. 
Define $(r, \theta)$ implicitly by 
\begin{eqnarray*}
      x~ -~ x_0 &=&  \sqrt{r^2+ t^2}  \cdot \cos(\theta)  \\
      y~ -~ y_0 &=&  r \cdot \sin(\theta)
\end{eqnarray*}
The slit then corresponds exactly to the locus $r=0$ and 
the level sets $r=const$ correspond to ellipses that surround the 
slit. If, in turn, one sets $r = t \sinh(z)$ then the equation
$\Delta \psi= E \cdot \psi$ is separable in the variables $(z, \theta)$. 
The solutions to the resulting ordinary differential equations 
are called Mathieu functions. Appendix \ref{MathieuAppendix} 
contains the basic facts about these functions that we use here. 
We note that the work of Y. Colin de Verdiere \cite{CdV} called
our attention to the usefulnes of these coordinates when considering
slits, a fact that mathematical physicists have long been aware
of \cite{MrsRbn} \cite{MorseFeshbach}. 

We use Theorem \ref{Main}  to extend the generic 
simplicity results of \cite{HJ1}. For example, we 
consider multiply connected polygons with $n$ exterior 
vertices.\footnote{One can regard such a polygon as the
result of removing finitely 
many simply connected polygons from a simply connected $n$-gon.}  
In \S 7, we prove the following:

\begin{thm}\label{NonSimplyConnected}
If $n \geq 4$, then almost every multiply connected polygon has 
simple Laplace spectrum. 
\end{thm}

In Theorem \ref{NonSimplyConnected} we assume that the 
boundary conditions are either Neumann on every edge or 
Dirichlet on every edge. We also consider `mixed' boundary
conditions in the sense that the condition on each boundary 
segment is either Dirichlet or Neumann. 

\begin{thm}\label{MixedDirichletNeumann}
For each combinatorial choice of mixed Dirichlet-Neumann boundary 
conditions, almost every simply connected polygon 
with $n \geq 4$ vertices has simple
Laplace spectrum. 
\end{thm} 

This latter theorem may be extended to multiply connected
 polygons with some mixed 
boundary conditions. It should be  noted, however, that our method 
doesn't give generic simplicity for every type of mixed boundary conditions 
in the case of multiply connected polygons. For instance, we cannot handle 
a slit with Dirichlet conditions 
on one side and Neumann conditions on the other.  

We now outline the contents of this paper. 
In \S 2 we precisely describe the eigenvalue problem 
on a slit domain and give a fuller description 
of elliptical coordinates.
In \S 3 we prove that sequences of eigenfunctions 
with uniformly bounded eigenvalues have convergent 
subsequences as the slit parameter $t$ tends to zero. 
In \S 4 we apply Kato's theory of analytic perturbations
to show that the eigenvalues and eigenfuntions depend
analytically on $t$.  We also derive variational formula
for these eigenvalues and prove that uniform estimates
on the logarithmic derivative can be `localized' to
a neighborhood of the slit. In \S 5 we prove a
convexity result for `radial' Mathieu functions.
In \S 6, we prove Theorem \ref{Main} according to the
strategy describe above. In \S 7, we precisely state
and then prove results concerning the spectral 
simplicity of polygons including Theorems 
\ref{NonSimplyConnected} and \ref{MixedDirichletNeumann}.


\section{The eigenvalue problem on a slit domain} \label{SectionProblem}
 
Let $\Omega \subset {\mathbb R}^2$ 
be a compact domain with Lipschitz boundary, and 
let $\Sigma$ be a disjoint union of line 
segments contained in the interior
of $\Omega$. The complement  $\Omega \setminus \Sigma$
is not a domain with Lipschitz boundary, and hence 
one must take care in defining the eigenvalue problem.   
The purpose of this section is to describe 
a resolution that will be useful to us. We refer the 
reader to \cite{Grisvard} for a complementary
discussion of the analysis of singular domains. 

To define the eigenvalue problem, we will regard 
a slit domain as the compact manifold with Lipschitz boundary
obtained by completing $\Omega \setminus \Sigma$
with respect to the length metric. To be more precise, 
we define the distance $d(x,y)$ between $x$ and $y$ to be 
the infimum of lengths of rectifiable paths 
$\gamma: [0,1] \rightarrow  \Omega \setminus \Sigma$
that join $x$ to $y$. Note that $d(x,y)$ equals
the  Euclidean distance between $x$ and $y$ 
iff $x$ and $y$ belong to a convex subset
of $\Omega \setminus \Sigma$.

\begin{defn}
The {\em slit domain} $\Omega_{\Sigma}$ 
is the metric completion of $\Omega \setminus \Sigma$ with
respect to $d$. We will refer to 
$\Omega_{\Sigma} \setminus (\Omega \setminus \Sigma)$
as the {\em slit}.
\end{defn}

Although $\Omega_{\Sigma}$ is not isometric to 
a subdomain of ${\mathbb R}^2$, it is naturally a 
compact Riemannian manifold with Lipschitz boundary. 
To see this, one can use elliptical coordinates 
to define a chart in a neighborhood of the slit.

\subsection{Elliptical coordinates in a neighborhood  of the slit}

Let $S^1 = {\mathbb R}/(2 \pi{\mathbb Z})$, and for each 
$t \geq 0$, define 
$\phi_t: {\mathbb R}^+ \times S^1 \rightarrow {\mathbb R}^2$
\[  \phi_{t}(r, \theta)~ 
  =~ \left( \sqrt{r^2+ t^2} \cdot \cos(\theta),~ 
        r\sin(\theta) \right).
\]  
This injective map sends each circle $\{r\} \times S^1$ 
onto an ellipse with foci at $(\pm t, 0)$. 
We have 
$\phi_t({\mathbb R}^+ \times S^1)
         ={\mathbb R}^2 \setminus \Sigma_t$
where $\Sigma_t= [-t, t] \times \{0\}$.
By regarding ${\mathbb R}^2 \setminus \Sigma_t$
as a subset of ${\mathbb R}^2_{\Sigma_t}$, 
the map $\phi_t$ extends to a smooth diffeomorphism 
from $[0, \infty) \times S^1$ onto ${\mathbb R}^2_{\Sigma_t}$. 
Abusing notation slightly, we will use $\phi_t$ 
to denote this extension.  

In these elliptical coordinates $(r, \theta)$ about 
the slit $\Sigma_t$, the gradient operator takes the form
\begin{equation} \label{Gradient}  \nabla_{t} w~ =~  
\frac{r^2+t^2}{r^2 + t^2\sin^2(\theta)} 
\cdot \frac{\partial w}{\partial r} \cdot \partial_{r}~ +~
      \frac{1}{r^2 + t^2 \sin^2(\theta) } \cdot 
 \frac{\partial w}{\partial \theta} \cdot \partial_{\theta},   
\end{equation}
and Lebesgue measure is expressed as
\begin{equation} \label{Lebesgue}   dm_{t}~ =~ 
  \frac{r^2 + t^2\sin^2(\theta)}{(r^2+t^2)^{\frac{1}{2}}}~ dr d \theta.
\end{equation}

The eigenvalue problem is unchanged if we rotate 
and/or translate the domain $\Omega$ and the slit  
$\Sigma \subset \Omega$ simultaneously.  
Thus, in the sequel, we will often make the following assumption.
\begin{ass}
We assume that $\Sigma = [-t,t] \times \{0\}$.
\end{ass}
Now choose $r_0>0$ so that the ball $B(\vec{0}, 2 r_0)$
is contained in $\Omega$. Then for $t< r_0$, the 
restriction of $\phi_t^{-1}$ to $[0, 2 r_0) \times S^1$ 
provides a chart in a neighborhood of the slit.

\subsection{The eigenvalue problem}

Let $dm$ denote Lebesgue measure. 
For each smooth function on the manifold $\Omega_{\Sigma}$
define 
\[ N(u)~  =~ \int_{\Omega} |u|^2~ dm \]
and 
\[ q(u)~  =~ \int_{\Omega} |\nabla u|^2~ dm. \]
Given  a measurable subset, $D$,
of the boundary of $\Omega_{\Sigma}$,  
define $H^1_D(\Omega_{\Sigma})$ to be the completion of
\[  \left\{ u \in C^{\infty}(\Omega_{\Sigma})~
  |~ u(m)=0,~\forall m \in   D \right\}
\]
with respect to the norm 
\[  u~ \mapsto~ q^{\frac{1}{2}}(u)~ +~  N^{\frac{1}{2}}(u). \]
The form $q$ extends to a closed quadratic form 
on $H^1_D(\Omega_{\Sigma})$ and the form 
$N$ extends to a closed quadratic form on $L^2(\Omega_{\Sigma}, dm)$. 
In the sequel we will let $\gotn(\cdot, \cdot)$ (resp.
$\gotq(\cdot, \cdot)$) denote the polarization of 
$N$ (resp. $q$).

A function $\psi \in H^1_D(\Omega_{\Sigma})$ is an 
{\em eigenfunction with eigenvalue $E$} if and only if 
\[    \gotq(\psi, v)~ =~ E \cdot \gotN(\psi,v) \]
for all $v \in H^1_D(\Omega_{\Sigma})$. Integration by parts
and standard elliptic estimates give that $\psi$ 
is smooth in $\Omega_{\Sigma}$ with $\Delta \psi = E \cdot \psi$ where $\Delta$
is the Laplacian acting on smooth functions on ${\mathbb R}^2$.
Moreover, we have 
\begin{enumerate}
\item $\psi(m)=0$ for all $m \in \mbox{Int}(D)$  (Dirichlet conditions), and 

\item $\nu(\psi)(m)=0$ for all  $m \in \mbox{Int}  (\partial \Omega_{\Sigma} \setminus D)$
    (Neumann conditions).
\end{enumerate}
Here $\nu$ denotes the outward normal derivative
along the boundary, $\partial \Omega_{\Sigma}$, of 
$\Omega_{\Sigma}$. Note that in the sequel $D$ will essentially be 
a union of segments.


\section{Convergence of  eigenfunctions with bounded eigenvalues}\label{SecOrdered}
 
Let $\Omega$ be a domain that contains the origin $\{0\}$,
and let $t_n>0$ be a sequence with $\lim_{n \rightarrow \infty} t_n =0$. 
For sufficiently large $n$, the segment $\Sigma_{t_n}=[-t_n,t_n] \times \{0\}$ 
lies in the interior of $\Omega$.  For each $n \in {\mathbb N}$,
let $D_n$ be a measurable subset of the boundary of the slit domain
$\Omega_{\Sigma_{t_n}}$ such that $D=D_n \cap \partial \Omega$
does not depend on $n$. Let $\psi_{n}$ be a normalized eigenfunction 
of $q$ on $H^1_{D_n}(\Omega_{\Sigma_{t_n}})$ with eigenvalue $E_{n}$. In this section
we prove the following:

\begin{thm} \label{OrderedConvergence}
If $E_n$ is a bounded sequence, then 
a subsequence $\psi_{n(k)}$ converges in $L^2(\Omega)$ 
to an eigenfunction $\psi$ of $q$ on $H^1_D(\Omega)$ with eigenvalue 
$E= \lim_{k \rightarrow \infty} E_{n(k)}$.  
Moreover, for every neighborhood $U$ of the slit
and every $j \in {\mathbb N}$,
the sequence $\psi_{n(k)}$ converges to $\psi$ in 
$C^j(\Omega \setminus U)$.  
\end{thm}

\begin{proof}
Let $U$ be a neighborhood of the slit. Each $\psi_n$ is an
eigenfunction of $\Delta$
with eigenvalue $E_n$, and hence for each $n$
such that $\Sigma_{t_n} \subset U$ we have 
\[  \int_{\Omega \setminus U} 
    \left|\psi_n \cdot \Delta^j \psi_n \right|~ dm~
   =~   E_n^j  \int_{\Omega \setminus U} |\psi_n|^2~ dm.
\]
Thus, using G\aa rding's inequality and the 
fact that $\psi_n$ is normalized, we find that
\[      \| \psi_n \|_j~ \leq~ C \cdot ( E_n^k+1) \] 
where  $\| \cdot \|_j$ is the norm associated to
the Sobolev space $H^j(\Omega \setminus U)$. 
Since by assumption, $E_n$ is a bounded sequence, we have that 
for each $j$, the sequence $n \mapsto \| \psi_n \|_j$ 
is bounded. 
 
Using a diagonalization argument, we find a function 
$\psi$ and a subsequence of $\psi_n$---still denoted $\psi_n$---such 
that for every $j$ and every neighborhood $U$
of the slit, $\psi_n$ converges to $\psi$ in the  
$H^j(\Omega \setminus U)$ norm. Thus, by the 
Sobolev embedding theorem, $\psi_n$ converges
to $\psi$ in $C^k(\Omega \setminus U)$ for every 
$k \in {\mathbb N}$.

By assumption, the $\psi_n$ have $L^2$-norm equal to 1. 
Thus, given $\epsilon >0$, Lemma \ref{NonConcentration} below
applies to give $N_1>0$ and $r^*$ so that if $n >N_1$, then 
\[   \int_B |\psi_n|^2~ dm~ \leq~ \frac{\epsilon}{6}. \]
where $B$ is some fixed ball centered at $0$ and included in the neighborhood 
$U_{t_n,r*}$ for $n>N_1$. 
Hence, by Fatou's Lemma,  we also have 
\[   \int_B |\psi|^2~ dm~ \leq~ \frac{\epsilon}{6}. \]
The functions $\psi_n$ converge uniformly to $\psi$ on the complement
of $B$, and hence there exists $N_2$ so that if $n>N_2$ then 
\[  \int_{\Omega \setminus U} |\psi-\psi_n|^2~ dm~ <~ \frac{\epsilon}{3}. \]
It follows that  $\psi_n$ converges to $\psi$ in $L^2(\Omega)$.
In particular, $\psi$ is $L^2$-normalized and hence 
is nontrivial. 

Finally, we show that $\psi$ is an eigenfunction. To do this we adapt 
an argument from \cite{CdV81}.  Define the distribution 
$T$ by $T(\phi) = \int_{\Omega} \psi \cdot \phi~ dm$. 
Then $\Delta' T \in H^{-2}(\Omega)$ where $\Delta'$ denotes
the distributional Laplacian. If the 
support of $\phi$ does not contain the origin, then 
$\Delta' T(\phi) = E \cdot T(\phi)$, and hence the 
singular support of $\Delta' T-ET$ is contained in $\{0\}$. 
Hence there exists $S \in H^{-2}(\Omega)$ with $\supp(S) \subset \{0\}$
such that 
\[ \Delta' T~ =~ E \cdot T~ +~ S. \]
It follows that there exists $C$ such that $S= C \cdot \delta$ 
where $\delta(\phi)= \phi(0)$. 

Let $G$ be the distribution defined by 
$G(\phi) = \int_{\Omega} L \cdot \phi~ dm$ where
$L(x,y)= \ln(\sqrt{x^2+ y^2})\chi(x,y)$ where $\chi$ is some smooth cut-off function 
near $0$ (say for instance in the ball $B$).
A direct computation shows that $\Delta' G - \delta$ is in $L^2(\Omega)$ and hence it follows from
above that $\Delta'(T-C \cdot G)$ is also in $L^2(\Omega)$ and hence, 
by elliptic regularity, $T-C \cdot G \in H^2(\Omega).$ 

It suffices to show that $C=0$.  On the one hand, by Fatou's Lemma, 
we have  
\[
  \int_{\Omega}  |\nabla \psi|^2~ dm~ \leq~ 
     \liminf_{n \rightarrow \infty}   \int_{\Omega}  |\nabla \psi_n |^2~ dm~ =~ E,
\]
and on the other hand 
\[ 
    \int_{\Omega}  |\nabla L|^2~ dm~  \geq~
   2 \pi  \int_0^{\epsilon} \left(\partial_r \ln(r)\right)^2~  r~ dr ~ =~ \infty \]
where $B(0, \epsilon) \subset \Omega$ is a ball centered at the origin.   
Therefore, $T-C \cdot G \notin H^1(\Omega)$ unless $C=0$. 
\end{proof}

\begin{lem}\label{NonConcentration}
Given $\epsilon >0$ and $E_0>0$, there exists $r^*>0$ such that 
if $u$ is an eigenfunction of $q$ on 
$H^1_D(\Omega_{\Sigma_{t}})$ with eigenvalue $E \leq E_0$, 
then for all $t < r^*$ we have
\[ \int_{U_{t,r^*}} |u|^2~ dm~ \leq~  
    \epsilon \int_{\Omega} |u|^2~ dm \]
where $U_{t,r^*}$ is the elliptical neighborhood of the
slit $\Sigma_t$ of radius $r^*$. 
\end{lem}

\begin{proof}
For positive $r$ and $\rho$, we have 
\begin{eqnarray*} |u(r+\rho, \theta)-u(r, \theta)|~ 
  & \leq & \int_{0}^{\rho} \left| \partial_ru(r+s, \theta) \right|~ ds \\
  & \leq & \sqrt{\rho} \cdot 
\left( \int_{0}^{\rho}
 \left| \partial_ru(r+s, \theta) \right|^2~ ds \right)^{\frac{1}{2}}
\end{eqnarray*}
by the Cauchy-Schwarz inequality. From this we find that
\begin{equation} \label{AfterCauchy}
 \frac{1}{2} \cdot  |u(r, \theta)|^2~ 
   \leq~  |u(r+ \rho, \theta)|^2~  +~
  \rho \int_{0}^{\rho} \left| \partial_ru(r+s, \theta) \right|^2~ 
  ds. 
\end{equation}
Given $r_1>0$, define $U_s= \{r~ |~ s \leq r \leq r_1+s\} \times S^1$.
Using  (\ref{Gradient}) and (\ref{Lebesgue}), we find that if $s \geq 0$
\begin{eqnarray*}  
 \int_{U_0} \left|\partial_r u(r+ s, \theta) \right|^2~  dm~
    &\leq& \int_{U_0} \left|\partial_r u(r+ s, \theta) \right|^2~
            (r^2+ t^2)^{\frac{1}{2}}~ d \theta dr~  \\
     &\leq&  \int_{U_0}|\partial_r u(r+ s, \theta)|^2~
            ((r+s)^2+ t^2)^{\frac{1}{2}}~ d \theta dr~  \\
    &=& \int_{U_s}|\partial_r u(r, \theta)|^2~
            (r^2+ t^2)^{\frac{1}{2}}~ d \theta dr. \\
    &\leq& \int_{U_s} |\nabla u(r, \theta)|^2~ dm.
\end{eqnarray*}
Thus, since $u$ is an eigenfunction
with eigenvalue $E$, 
\begin{equation} \label{GradientPart}
 \rho \int_{0}^{\rho} \int_{U_0} \left| \partial_ru(r+s, \theta) \right|^2~ dm~ ds~
 \leq~   \rho^2 E \int_{\Omega} |u|^2~ dm. 
\end{equation}

Fix some $\rho$ so that $\rho^2 E \leq \epsilon/4$,
and choose $r^*\leq \rho $.   
Uniformly for $r,t \leq r^*$, we have then 
$$
\frac{r^2+t^2\sin^2 \theta}{(r^2+t^2)^{\frac{1}{2}}},\leq\, (r^2+t^2)^{\frac{1}{2}}\,\leq\, r^*
$$
and
$$
\frac{(r+\rho)^2+t^2\sin^2\theta}{(r+\rho)^2+t^2)^\frac{1}{2}}\,
\geq \, 
\frac{(r+\rho)^2}{(2(r+\rho)^2)^{\frac{1}{2}}}
\geq \frac{\rho}{\sqrt2}
$$
%
%
%
from which it follows that
\[   \frac{r^2+ t^2 \sin^2 \theta}{(r^2+ t^2)^{\frac{1}{2}}}~
 \leq~ \frac{ r^*\sqrt{2}}{\rho} \cdot
  \frac{(r+\rho)^2+ t^2 \sin^2 \theta}{((r+\rho)^2+ t^2)^{\frac{1}{2}}}.
\]
We now fix $r^*\leq \frac{\ep \rho}{\sqrt{32}}$ so that we have
\[   \frac{r^2+ t^2 \sin^2 \theta}{(r^2+ t^2)^{\frac{1}{2}}}~
 \leq~ \frac{\ep}{4} \cdot
  \frac{(r+\rho)^2+ t^2 \sin^2 \theta}{((r+\rho)^2+ t^2)^{\frac{1}{2}}}
\]
uniformly for $r,t\leq r^*.$ 

Thus,  using (\ref{Gradient}) and (\ref{Lebesgue}), we have
\begin{eqnarray*}   \int_{U_0} |u(r+ \rho, \theta)|^2~ dm~
    &\leq& \frac{\epsilon}{4} \int_{U_0}|u(r+ \rho, \theta)|^2~
           \frac{(r+\rho)^2+ t^2 \sin^2 \theta}{((r+\rho)^2+ t^2)^{\frac{1}{2}}}~ 
              d \theta dr~  \\
    &=&  \frac{\epsilon}{4} \int_{U_{\rho}} |u(r, \theta)|^2~
       \frac{r^2+ t^2 \sin^2 \theta}{(r^2+ t^2)^{\frac{1}{2}}}~ d \theta dr \\
    &\leq&  \frac{\epsilon}{4} \int_{\Omega} |u(r, \theta)|^2~ dm. 
\end{eqnarray*}
By combining this with (\ref{AfterCauchy}) and (\ref{GradientPart})
we find that 
\begin{eqnarray*} 
 \frac{1}{2} \int_{U_{0}} |u(r, \theta)|^2~ dm~
  &\leq &   \left(\frac{\epsilon}{4}~ +~  \rho^2  E \right)~  
 \int_{\Omega} |u(r, \theta)|^2~ dm.~ 
\end{eqnarray*}
The claim follows since $\rho^2 E \leq \epsilon /4$. 
\end{proof}


\section{Analyticity, variational formulae, and localization}

In this section we use standard perturbation theory 
to show that the eigenvalues and eigenfunctions
depend analytically on $t$ for $t>0$. 
We then derive basic formulae relating 
the derivative of an eigenbranch to the derivatives of the
quadratic forms $q$ and $N$. We show that for small 
$t$ the dominant terms in the derivative fomulae
can be localized to a neighborhood of the slit. 
Finally, we evaluate these local formulae on an
elliptical neighborhood of the slit.

\subsection{Pulling back to elliptical coordinates}

Standard analytic perturbation theory \cite{Kato} applies to a family
of quadratic forms on a {\em fixed} Hilbert space.  For this reason 
we will modify the family quadratic forms $q$ and $N$ but in
such a way that we obtain an equivalent eigenvalue problem.  

Of the various possible approaches, we choose to modify the map
$\phi_t$  so that the inverse image of $\Omega_{\Sigma_t}$ is constant
in $t$. The fixed Hilbert space will then be $L^2$ on this inverse
image with respect to the pull-back of Lebesgue measure.

Let $r_0$ be as in \S \ref{SectionProblem}. Namely, 
$B(\vec{0},2r_0) \subset \Omega$. 
Let $\sigma: {\mathbb R} \rightarrow {\mathbb R}$
be a smooth, positive, decreasing function such that 
$\sigma(t)=1$ if $|t| \leq 1$ and  $\sigma(t)=0$ if $|t| \geq 2$. 
Define 
$\tilde{\phi}_{t}: {\mathbb R}\times S^1 \rightarrow  
{\mathbb R}^2$  by 
  \[  \tilde{\phi}_{t}(r, \theta)~ 
  =~ \left( \sqrt{r^2+ t^2 \cdot \sigma(r/r_0)} 
        \cdot \cos(\theta),~ 
        r\sin(\theta) \right).
\]  
By construction, $\tilde{\phi}_{t} $ defines a smooth diffeomorphism
from $[0, \infty) \times \R$ onto $\R^2_{\Sigma_t}$. 

If $t < r_0$, then the set $M=\tilde{\phi}_{t}^{-1}(\Omega_{t})$ 
does not depend on $t$.  We pull-back the Dirichlet energy functional 
and the $L^2$-norm on $\Omega$ to functionals on $M$.
 
For computational convenience, we express the change 
of variables in the language of Riemannian geometry. 
For each $t$, define the Riemannian metric
\[   g_{t}~ =~ \tilde{\phi}_{t}^*(dx^2+ dy^2) \]
on $[0, \infty) \times S^1$.  In particular, $\tilde{\phi}$
is a Riemannian  isometry from $(\Omega_{t}, dx^2+ dy^2)$ onto
$(M,  g_{t})$.

Let $dm_{t}$ denote the Riemannian measure:
$dm_{t}= (\tilde{\phi}_t)^{-1}_{*}(dx dy)$. The space, $L^2(M)$, of 
functions that are square 
integrable with respect to $dm_{t}$ does not depend
on $t$.  On the other hand, the associated norm does
depend on $t$.  Let $N_t$ denote the quadratic form\footnote{This
quadratic form is defined on functions on $M$, and hence is not
the same as the quadratic form $N$ defined in \S \ref{SectionProblem}.
On the other hand, these forms differ by pull-back by $\phi_t$,
and the context will make clear which is being used.}
\begin{equation} \label{Nt}
   N_{t}(u)~ =~ \int_{M}  |u|^2~ dm_{t}.  
\end{equation}
In the sequel we will let $\gotN_{t}(\cdot, \cdot)$ 
denote the polarization of $N_{t}$.

Let $\nabla_{t}$ denote the Riemannian gradient for $(M, g_t)$.
Namely,  we have $g_{t}(\nabla_{t} u, X)=  du(X)$ for every smooth function
$u: M \rightarrow \R$ and vector field $X$ on $M$. 
For smooth functions $u: M \rightarrow \R$  define 
\begin{equation} \label{q}
  q_{t}(u)~ =~  
\int_{M} g_{t}(\nabla_{t} u, \nabla_{t}u)~ 
    dm_{t}. 
\end{equation}

Given a measurable subset $D$ in the boundary of $M$, 
define $H^1_D(M)$ to be the completion with respect to the norm
\[ 
u \mapsto q_{t}(u)^{\frac{1}{2}} + N_{t}(u)^{\frac{1}{2}}.
\]
of the set of smooth 
functions on $M$ that vanish on $D$.
Since the equivalence class of this norm
is independent of $t>0$, the completion
does not depend on $t$.
The form $q_{t}$ extends to a closed quadratic
form on $H^1_D(M)$.   In the sequel we will let 
$\gotq_{t}(\cdot, \cdot)$ denote the polarization of $q_{t}$. 

\begin{prop} \label{Equivalent}
$\psi$ is an eigenfunction of the
quadratic form $q$ on $H^1_{\phi_t(D)}(\Omega_{\Sigma_t})$ with respect to 
$N$ on $L^2(\Omega_{\Sigma_t})$ with eigenvalue $E$ 
if and only if $u=\tilde{\phi}^*_{t}(\psi)$ is a eigenfunction 
of $q_t$ on $H^1_D(M)$ with respect to $N_t$ on $L^2(M)$
with eigenvalue $E$. 
\end{prop}

\begin{proof}
This follows from a straightforward 
accounting of the change of variables given by $\tilde{\phi}_t$.
\end{proof}

\subsection{Analyticity of eigenvalues and eigenfunctions}

\begin{prop} \label{Analytic}
The eigenvalues and eigenfunctions of $q_t$ with respect to $N_t$
vary real-analytically. To be precise: 
$\forall k \in {\mathbb N}$,  $\exists$
real-analytic paths $\psi_k: (0, t_0]
\rightarrow H^1_D(M)$ and
$E_k:(0, t_0] \rightarrow \R$ such that

\vspace{.2cm}

\begin{enumerate} 
\renewcommand{\labelenumi}{(\alph{enumi})}
\item $\gotq_{t} \left(\psi_k(t),v \right) =
    E_{k}(t) \cdot \gotN_{t} \left(\psi_k(t), v \right)$ 
  for all     $v \in H^1_D(M)$.

\vspace{.3cm}

\item $N_{t}(\psi_k(t))=1$, and  

\vspace{.3cm}

\item the span of $\{ \psi_k(t)~ |~  k \in {\mathbb N} \}$ 
      is dense in $L^2(M)$.
\end{enumerate}  
\end{prop}

\begin{proof}
Apply standard analytic perturbation theory. 
In particular, both $q_{t}$ and $N_{t}$ are holomorphic 
families of quadratic forms of 
type (a) as in \S VII.4 \cite{Kato}.
The eigenvalue problem is of  generalized form 
as in \S VII.6 \cite{Kato}. 
(See especially Remark VII.6.2 \cite{Kato} and the 
discussion on page 419.)
The analogue of Remark VII.4.22 \cite{Kato} gives the claim.  
\end{proof}

\begin{remk} 
We will use the expression {\em eigenbranch} to 
designate the mapping $t\rightarrow (\psi_k(t),E_k(t))$ 
and will use the expression 
{\em normalized eigenbranch} if, in addition, 
 $N_{t}\left(\psi_k(t)\right)=1.$ 
By extension, the term {\em eigenbranch} may also refer to 
either the eigenvalue or the eigenvector singly. Since we will 
be dealing with one eigenbranch at a time, the index $k$ will 
be systematically dropped. 
\end{remk}  

\begin{coro}
The Neumann (resp. Dirichlet) 
eigenvalues of the Laplacian on $\Omega_{\Sigma_t}$
depend analytically on $t \in (0, t_0]$.
\end{coro} 

\begin{proof} 
 Combine Proposition \ref{Equivalent}, Proposition
 \ref{Analytic}, and the discussion in \S \ref{SectionProblem}.
\end{proof}

\subsection{Variational formulae}

\begin{nota}
A dot above a quantity will
indicate differentiation with respect to $t$. 
For example, $\dot{E}$ indicates the first derivative
of an eigenvalue branch $E_t$. In what follows,   
we will often suppress the dependence of $q$, 
$\psi$, and $E$ on $t$ from the notation.
\end{nota} 

We begin with a well-known, general variational formula.
\begin{prop} \label{Variation}
We have 
\[ \dot{E} \cdot N(\psi)~ 
     =~ \dot{q}(\psi)~ -~ E \cdot \dot{N}(\psi). \]
\end{prop}
\begin{proof}
Substitution of $\dot{\psi}$ for $v$ in part (a) of Proposition \ref{Analytic}
gives
\begin{equation} \label{DotSub}
 \gotq(\psi, \dot{\psi})~ =~ E \cdot \gotN(\psi, \dot{\psi}).
\end{equation}
By differentiating part (a) of Proposition \ref{Analytic}, we obtain
\[  \dot{\gotq}(\psi, v)~ +~ \gotq(\dot{\psi},v)~ =~ \dot{E} \cdot \gotN(\psi,v)~
   +~ E \cdot \left( \dot{\gotN}(\psi,v)~ +~ \gotN(\dot{\psi}, v) \right) \]
for all $v$.   By substituting $\psi$ for $v$ 
and using (\ref{DotSub}), we find that
\[  \dot{q}(\psi, \psi)~ =~ \dot{E} \cdot N(\psi)~
   +~ E \cdot \dot{N}(\psi). \]
\end{proof}

In the case that $q$ is the Dirichlet energy associated
to a family of Riemannian metrics, the quantities in 
Proposition \ref{Variation} can be expressed in terms 
of the first variation of the metric and the associated
Riemannian measure. 

\begin{prop}\label{Dotq1}
Let $t \rightarrow g_{t}$ be a real-analytic family 
of Riemannian metrics on $M$, and suppose that $q$ is 
defined by (\ref{q}). Then we have 
\[  \dot{q}(u)~ =~  - \int_M \dot{g}
       \left (\nabla u, \nabla u \right)~ dm~
         +~ \int_M g \left(\nabla u, \nabla u \right)~ \dot{dm} \]
and 
\[  \dot{N}(u)~ =~  \int_M |u|^2~ \dot{dm}. \]
\end{prop}

\begin{proof}
By differentiating (\ref{q}) we have 
\[    \dot{q}(u)~ =~ 
   \int_M \dot{g} \left (\nabla u, \nabla u \right)~ dm~ +~
      2\int_M g \left (\nabla u, \dot{\nabla} u \right)~ dm~
         +~    \int_M g \left(\nabla u , \nabla u \right)~ \dot{dm}. \]
By definition $g(\nabla u, X)= X(f)$ for all vector fields $X$,
and hence we have 
\begin{equation}\label{gDotNabla}  \dot{g}(\nabla u, X)~ +~ g( \dot{\nabla} u, X)~ =~ 0. \end{equation}
In particular, if $X= \nabla u$, then we have 
\[    \int_M \dot{g}(\nabla u, \nabla u)~ dm~
 +~  \int_M g( \dot{\nabla} u, \nabla u)~ dm~=~0. \]
Substitution into the formula for $\dot{q}$ gives the first formula.
The second formula follows from (\ref{Nt}).  
\end{proof}

The following lemma will be used to translate 
estimates on the logarithmic derivative of an 
eigenbranch into statements concerning the convergence 
of the eigenbranch.  

\begin{lem}\label{ExpTrick}
If there exists $t_0$ and  a continuous positive  function 
$\rho:(0, t_0] \rightarrow {\mathbb R}^+$ such that 
\[ \partial_t E_t~ \geq~  -~ \rho(t) \cdot E_t
\]
for all $t \in (0, t_0]$,
then the function 
\[   F(t)~  =~ \exp \left(-\int_{t}^{t_0} \rho(s)~ ds \right) \cdot E_t
\]
converges as $t$ tends to zero. In particular, if $\rho$
is integrable, then $E_t$ converges 
as $t$ tends to zero. 
\end{lem} 

\begin{proof}
We have
\[ F'(t)\,=\,\left( \partial_t E_t~ +~ \rho(t) \cdot E\right) \cdot 
\exp \left(-\int_{t}^{t_0} \rho(s)~ ds \right)~ \geq 0, \]
and hence $F$ is increasing. Since $F$ is nonnegative, the claim follows. 
\end{proof}

Using this lemma and Proposition \ref{Variation} we have the following corollary. 
\begin{coro}
Suppose that there exists a constant $C$ such
that for all $u \in H^1(M)$  we have 
\begin{equation} \label{logq} 
   \dot{q}(u)~ \geq~ -C \cdot q(u)
\end{equation}
and 
\begin{equation} \label{logN}
    {\dot{N}(u)}~ \leq~ C \cdot N(u). 
\end{equation}
Then 
$E_t$ converges as $t$ tends to $0.$
\end{coro}

\begin{proof}
From Proposition \ref{Variation} we have
$\dot{E}\geq -2C \cdot E$. Thus we can apply 
Lemma \ref{ExpTrick} with $\rho \equiv 2 C$.
\end{proof}

\subsection{A localization principle}

Unfortunately, inequalities  (\ref{logq}) and (\ref{logN})
do not hold true for the singular perturbation that we consider here. 
We now show, however, that these inequalities do hold 
for all $u$ with support outside of a neighborhood of the
slit (Proposition \ref{SupportAway}). 

To state this result in a convenient form, we introduce
the following notation. 
Let $U$ be a measurable set.\footnote{In the sequel, 
$U$ will often be an elliptical neighborhood of the slit.}  
For $w \in H^1$, we denote by 
$q_U(w)$ the `restriction' of $q$ to $U$. That is, 
\[
q_U(w)\,=\, \int_U |\nabla w|^2\,dm.
\]
We will use analogous notation for the quadratic forms 
$\dot{q}$, $N$, and $\dot{N}$.

\begin{prop} \label{SupportAway}
Let $U \subset \Omega$ be a neighborhood of the slit. 
There exists a constant $C_U$ such that 
for any $w\in H^1$ 
\[
   {\dot{q}_{M\backslash U}(w)}~ \geq~ -C_U \cdot {q(w)}
\]
and 
\[
    \dot{N}_{M\backslash U}(w)~ \leq~ C_U \cdot N(w). 
\]
\end{prop}

\begin{proof}
Let $SM$ denote the unit tangent bundle to $M$ (with respect to $g_{t_0})$, and 
for each small $t$, define  $F_t: SM \rightarrow \R$ by
\[   F_t(X)~ =~  \frac{  \frac{d}{d t} g_{t}(X,X)}{ g(X,X)}. \]
The restriction of $g_{t}$ to $\Omega \setminus U$ is real-analytic 
for $t \in [0, t_0]$, and hence  
\[   C_1~ =~ \sup \{ F_t(X)~ |~  t \in [0, t_0] \mbox{\ and \ } X \in S(\Omega \setminus U) \}
\]
is finite. By homogeneity of $g$ and $\dot{g}$ in each tangent space, 
we have $\dot{g}(X,X) \leq C_1 \cdot g(X,X)$ for 
all $X \in T(\Omega \setminus U)$.  Therefore, 
\[ 
 \int_{M\backslash U} \dot{g} \left(\nabla w, \nabla w \right)~ dm~
    \leq~ C_1 \int_{M\backslash U}   g \left(\nabla w, \nabla w \right)~ dm~
\]

Let $G_t: [0, r_0] \times M \rightarrow \R$ be defined by
\[   G_t(p)~ =~  \frac{\frac{d}{d t} dm_{t}}{dm_{t}}(p). \]
Since $dm_{t}$ restricted to $\Omega \setminus U$ depends 
real-analytically on $t \in [0, t_0]$,
\[   C_2~ =~ \sup \{ |f(t, X)|~ |~  t \in [0, r_0] \mbox{\ and \ } X \in \Omega \setminus U \}
\]
is finite.  It follows that 
\[   \int_{M\backslash U} w^2~ \dot{dm}~
\leq ~ C_2 \int_{M\backslash U}   w^2~ dm,  
\] 
and 
\[ 
 \int_{M\backslash U} 
g\left( \nabla w,\nabla w\right)\, \dot{dm}~
\geq~ -C_2\,\int_{M\backslash U} g\left( \nabla w,\nabla w \right)\, dm. 
\]
We use Proposition \ref{Dotq1} and set $C_U =C_1 + C_2.$ 
The result then follows. 
\end{proof}

\begin{coro}   \label{GlobalReduction}
Let $U$ be a neighborhood of the slit.
Suppose that there exist functions 
$\alpha, \beta, \gamma:[0, \infty) \rightarrow [0, \infty)$
such that the following holds for an eigenbranch $\psi$ 
\[  \dot{q}_U(\psi)~
    \geq~ -\alpha(t) \cdot q_U(\psi)~ -~ \beta(t) \cdot
    E \cdot N_U(\psi),   \]
and 
\[  \dot{N}_U( \psi)~
    \leq~ \gamma(t) \cdot N_U(\psi). \]
Then there exists $ C' > 0$ such that 
\[  \dot{E}~ \geq~ -\left( C'~ -~ \alpha(t)~ -~ \beta(t)~ -\gamma(t)\right) E.\]
\end{coro}   
\begin{proof} 
We have
\begin{eqnarray*}
  \dot{q}(\psi) &=& \dot{q}_{M\backslash U}(\psi)~ +~ \dot{q}_U(\psi)  \\
   &\geq& -C \cdot q(\psi)~ -\alpha \cdot q_U(\psi)~ 
     -~ \beta \cdot E \cdot N_U(\psi)  \\
     &\geq& -(C+ \alpha) \cdot q(\psi)~ -~
      \beta \cdot E \cdot N( \psi) \\   
   &\geq& -(C+ \alpha+ \beta) \cdot E \cdot N(\psi)
\end{eqnarray*}
where the constant $C$ comes from  Proposition  \ref{SupportAway}. 
Similarly, we find that 
\[  \dot{N}_U( \psi)~  \leq  (C + \gamma) \cdot N_U(\psi). \]
The estimate then follows from 
Proposition \ref{Variation} with $C'=2C$.
\end{proof}

\begin{remk}
Proposition \ref{SupportAway} and Corollary \ref{GlobalReduction}
represent a {\em localization principle} in the sense that 
an estimate of the logarithmic derivative of $E_t$ 
has been reduced to the study of the functionals $\dot{q}$
and $\dot{N}$ in a small neighborhood of the slit. 
This principle applies more generally to singular perturbation
problems in which the `singular support' of the perturbation 
is small. 
\end{remk}

\subsection{Evaluation on the ellipse}

We now let $U$ be the elliptical neighborhood of 
the slit of radius $r_0$ and evaluate $q_U$ and $N_U$. 
By using the expression for the gradient (\ref{Gradient}) 
and the Lebesgue measure (\ref{Lebesgue}) in elliptical
coordinates, we find that
\[
  q_U(v)~ =~ \int_{U}
    |\partial_r v |^2~ (r^2+t^2)^{\frac{1}{2}}~ dr d \theta~
    +~  \int_U ~|\partial_{\theta} v|^2~
      \frac{dr d \theta}{(r^2+t^2)^{\frac{1}{2}}} 
\]
and 
\[
  N_U(v)~ =~   \int_U
    |v|^2 
  \frac{\left(r^2 +t^2\sin^2(\theta) \right)}{
    (r^2+t^2)^{\frac{1}{2}} }~ dr d \theta.
\]
By Proposition \ref{Dotq1} or direct 
computation, we find that 
\[ 
  \dot{q}_U(v)~ =~ 
t  \int_{U}
    |\partial_r v |^2~  \frac{dr d \theta}{(r^2+t^2)^{\frac{1}{2}}}~
    - ~ t  \int_U ~|\partial_{\theta} v|^2~
      \frac{dr d \theta}{(r^2+t^2)^{\frac{3}{2}} } 
\]
and 
\[
       \dot{N}_U(v)~
   =~  t \int_U
    |v|^2  \frac{\left((2 \sin^2(\theta)-1)r^2 +t^2\sin^2(\theta) \right)}{
    (r^2+t^2)^{\frac{3}{2}} }~ dr d \theta.
\]

It will prove convenient to make a change of variables.
\begin{nota} 
In the following, we let
\begin{itemize}
 \item  $r=t \cdot  \sinh x$, 

 \item $Y_t=\sinh^{-1}(r_0/t)$, and 

 \item $U_t=[0, Y_t] \times S^1$.
\end{itemize}
\end{nota}

With this change of coordinates, the formulae above become 
\begin{equation} \label{qEllipse}
  q_U(v)~ =~  \int_{U_t}
    |\partial_x v|^2~ dx d \theta~
    +~  \int_{U_t} ~|\partial_{\theta} v|^2~
      dx d \theta,
\end{equation}
\begin{equation}  \label{NEllipse}
  N_U(v)~ =~   t^2 \int_{U_t}
    |v|^2 
  \left(\sinh^2 x + \sin^2(\theta) \right)~ dx d \theta,
\end{equation} 
\begin{equation} \label{qDot}
  \dot{q}_U(v)~ =~ 
  \frac{1}{t}  \int_{U_t}
    \frac{|\partial_x v |^2}{\cosh^2 x}~ dx d \theta~
    -~ \frac{1}{t}  \int_{U_t}~
    \frac{|\partial_{\theta} v|^2}{\cosh^2 x}~ dx d \theta, 
\end{equation}
and 
\begin{equation}  \label{NDot}
       \dot{N}_U(v)~
   =~  t \int_{U_t}
    |v|^2  \left( \sin^2(\theta)~ - \cos^2(\theta) \tanh^2(x)
    \right)~ dx d \theta.
\end{equation}
%

From (\ref{qDot}) and (\ref{NDot}) we find that
\begin{equation} \label{qDotEst}
  \dot{q}_U(v)~ \geq~ 
    -~  \frac{1}{t}  \int_{U_t}~
    \frac{|\partial_{\theta} v|^2}{\cosh^2 x}~ dx d \theta,
\end{equation}
and
\begin{equation}  \label{NDotEst}
       \dot{N}_U(v)~
   \leq~  t \int_{U_t}
    |v|^2  \sin^2(\theta)~ dx d \theta~
\leq~  t \int_{U_t}
    |v|^2 ~ dx d \theta~
\end{equation}


\section{Mathieu functions and integral estimates} 
\label{SectionSeparation}

In this section we provide estimates that will be used 
to control ratios such as $\dot{N}/N$ and $\dot{q}/q$ in the
next section.

\subsection{Separation of variables}

E. Mathieu \cite{Mathieu} observed that one can 
apply the method of separation of variables to 
the eigenvalue problem, $\Delta \psi= E \cdot \psi$,
on an ellipse in ${\mathbb R}^2$.  He made a 
detailed study of the solutions to the resulting ordinary 
differential equations, solutions that now bear his name. 
To perform the separation of variables,
we use the same change of variables as in the preceding section, 
i.e. we set 
\[   r~ =~ t \cdot \sinh(x). \]

Let $h \in {\mathbb R}$.  The operator 
\[  -\frac{\partial^2}{\partial \theta^2}~ +~ h^2 \cos^2(\theta) \]
acting on $L^2({\mathbb R}/(2 \pi {\mathbb Z}), d \theta)$
has discrete spectrum
\[ 0~ \leq~ b_0(h)~ \leq~ b_1(h)~ \leq~ b_2(h)~ \leq~ \cdots.  \]
For $h \neq 0$, the spectrum is simple \cite{MorseFeshbach}. 
Let $v_{i,h}$ denote the $L^2$-normalized eigenfunction 
associated to $b_i(h)$. We will call $v_{i,h}$ 
the {\em $i^{{\rm th}}$ angular Mathieu function}. 
In the following we will often suppress the dependence 
of $v_{i,h}$ on $h$ from the notation. 

For the convenience of the reader, we prove the following 
in Appendix \ref{MathieuAppendix}. 

\begin{prop}
Suppose that $\Delta \psi = E \cdot \psi$ on 
the ellipse of radius $r_0=t \sinh(x_0)$, and suppose that $\psi$
satisfies a Dirichlet (resp. Neumann) boundary condition 
on the slit: 
For $0 \leq \theta  \leq 2\pi$ we have  
\[ \psi(0, \theta) \equiv 0 \ \ \   (\mbox{resp.  \ } 
   \partial_x \psi(0, \theta) \equiv 0). \] 
Then for $x \geq 0$,
\begin{equation} \label{Expand}
 \psi(x, \theta)~ =~ \sum_i u_{i}(x) \cdot v_{i}(\theta)
\end{equation}
where 
\[     h~ =~ t \cdot \sqrt{E}, \]
and where $u_{i}:[0, x_0) \rightarrow {\mathbb R}$ 
is a solution to
\begin{equation}  \label{UODE}  
-~ u''(x)~ 
    +~ \left(  b_{i}(h)~ -~ h^2 \cosh^2(x) \right) 
     \cdot u(x)~ =~ 0
\end{equation}
with $u(0)=0$ (resp. $u'(0)=0$). 
\end{prop}

A solution $u$ to (\ref{UODE}) will be called a 
{\em radial Mathieu function}.

In the sequel, our methods will rely upon 
estimates of the distribution of the $L^2$ and $H^1$ 
mass of an eigenfunction on the ellipse.  The following
lemma will allow us to reduce such estimates to
estimates of radial Mathieu functions.  
     
\begin{lem}  \label{Decomp}
We have 
\begin{equation}   \label{Decompweight}
\int_{S^1} |\psi|^2~ d\theta~
 =~ \sum_{i=0}^{\infty}~ |u_i(x)|^2.
\end{equation}
\begin{equation} \label{DecompweightTheta}
  \int_{S^1} \left|\partial_{\theta}\psi \right|^2 d \theta~  
+~ h^2\int_{S^1} |\psi|^2~ \cos^2(\theta)~ d \theta~
=~  \sum_{i=0}^{\infty}~ b_i(h) \cdot  |u_i(x)|^2.
\end{equation}
\end{lem}

\begin{proof}
By standard Sturm-Liouville theory, the angular Mathieu 
functions $\{v_i \}$ form a complete and orthogonal 
set in $L^2(S^1, d \theta)$. 
(See Appendix \ref{MathieuAppendix}).  
Thus, (\ref{Decompweight}) follows from (\ref{Expand}).

From (\ref{Expand}), we have $\partial_{\theta} \psi = 
\sum u_i \cdot \partial_{\theta}v_i$. Since $v_i$ is an 
eigenfunction with eigenvalue $b_i= b_i(h)$ we have 
\[  -\partial^2_{\theta} v_i~ + h^2 \cos^2(\theta) \cdot v_i~ 
    =~ b_i \cdot v_i.
\]
By multiplying by $v_j$ and integrating by parts, 
we have
\[  \int_{S^1}  \partial_{\theta} v_i \cdot
  \partial_{\theta} v_j~ d \theta~
    +h^2   \int_{S^1} v_i~ \cdot v_j~ \cos^2(\theta)~ d \theta~
  =~  b_i  \int_{S^1} v_i~ \cdot v_j~ d \theta.
\]
Equation (\ref{DecompweightTheta}) then follows from 
the fact that $\{v_i\}$ is complete and orthogonal. 
\end{proof}


\subsection{Radial convexity estimates} \label{ConvexitySection}

Our estimates of the distribution of the $L^2$ mass 
of a radial Mathieu function $u$ depend upon the following 
observation. 
We have 
\[  (u^2)''~  =~ 2 \cdot u \cdot  u''~ +~ 2 \cdot (u')^2~ \geq~  
    2 \cdot u \cdot  u'', \]
and hence if $b_i-h^2 \cosh^2(x) \geq \frac{1}{2}$, then 
by (\ref{UODE}) we have 
\[  (u^2)''~ \geq~  u^2. \]

We will let $w:[0, X] \rightarrow {\mathbb R}^+$
denote a smooth function such that
\[ w''(x)~ \geq~  w(x), \]
\[w(x)~  \geq~ 0 \]
for all $x \in [0,X]$ and 
\[ w'(0)~ \geq~ 0. \]
In particular, if $X$ satisfies 
\begin{equation} \label{ConvexCondition}
 b_i~ -~ h^2 \cdot \cosh^2(X)~ \geq~ \frac{1}{2},  
\end{equation}
then the square, $u_i^2$, of a radial Mathieu function
satifying either Dirichlet or Neumann conditions is 
an example of such a function $w$.

The following expression of convexity is the basis for
our estimates. 

\begin{lem} \label{SuperAdditive}
For all $x, y\geq 0$  such that $x + y \leq X$, we have
\[   w(x+y)~ \geq~  w(x) \cdot  \cosh(y). \]
\end{lem}

\begin{proof}
The claim holds if $w(x)=0$. So we may assume that 
 $w(x) >0$. Let $z(y) = w(x+y)/w(x)$. 
Note that $z(0)=1$ and  $z'(y) \geq 0$. Since 
$w''(x+y) \geq w(x+y)$, we have  
$\frac{d^2}{dy^2} (z(y)-\cosh( y)) \geq 0$ and
since $w'(x) \geq 0$, we have
$\left. \frac{d}{dy} \right|_{y=0}  ( z(y)-\cosh(y)) \geq 0$.
It follows that  
$\partial_y (z(y)-\cosh(y)) \geq 0$ for all $y \geq 0$. 
Since $z(0)-\cosh(0)=0$, we have 
$z(y)-\cosh(y)\geq 0$ for all $y$, and the result follows.
\end{proof}

\begin{prop} \label{Convex}
Let $p:[0,X] \rightarrow {\mathbb R}^+$ be a decreasing 
integrable function. Then  
\[ \int_{0}^{X}~ p(x) \cdot w(x)~ dx~ \leq~ 
   \left( \frac{p(0)}{\cosh \left(X/2 \right)}~
  +~ p(X/2) \right)~ 
       \int_{0}^X w(x)~ dx.
\]
\end{prop}

\begin{proof}
Applying Lemma \ref{SuperAdditive} with $y=X/2$ gives 
\[ \cosh(X/2) \int_{0}^{X/2} w(x)~ dx~ 
  =~ \int_{0}^{X/2} w(x+ X/2)~ dx~ \leq~
 \int_{X/2}^X  w(x)~ dy
\]
and hence
\begin{equation} \label{half}
 \int_{0}^{X/2} w(x)~ dx~ \leq~
 \frac{1}{\cosh(X/2)} \int_{X/2}^X  w(x)~ dy.
\end{equation}
Since $p$ is decreasing
\[  \int_0^X p(x) \cdot w(x)~ dx~ 
\leq~ p(0)~ \int_{0}^{X/2} w(x)~ dx~ 
+~  p(X/2)~ \int_{X/2}^{X} w(x)~ dx. 
\]
Combining this with (\ref{half}) gives the claim. 
\end{proof}

\begin{prop} \label{Convex2}
Let $p:[0,X] \rightarrow {\mathbb R}^+$ be an increasing 
integrable function. Then  
\[ \int_{0}^{X}~  w(x)~ dx~ \leq~\frac{2}{p(X/2)}
       \int_{0}^X w(x) \cdot p(x)~ dx.
\]
\end{prop}

\begin{proof}
From (\ref{half}) we have 
\[ \int_0^X w(x)~ dx~ \leq~ \left(\frac{1}{\cosh(X/2)}+ 1 \right) 
   \int_{X/2}^X w(x)~ dx. \]
Since $p$ is increasing, we have 
\[    \int_{X/2}^X w(x)~ dx~ \leq~ \frac{1}{p(X/2)} 
   \int_{X/2}^X w(x) \cdot p(x)~ dx.
\]
By combining these inequalities and using the fact
that $\cosh(X/2) \geq 1$, we obtain the claim. 
\end{proof}


\section{Limits for  analytic eigenbranches} \label{SecLimits}

In this section we prove that each real-analytic eigenvalue
branch $E_t$ converges. The proof consists of three main steps.
First we prove that $t^2 \cdot E_t$ converges as $t$ tends to zero. 
We use this to then prove that $t^{2k} E_t$ converges for some $k<1$.
Finally, we use this to prove that $E_t$ converges. At each new stage, the 
previous estimate is used to control $h(t)$ along the eigenbranch.

\subsection{Convergence relative to $t^2$}

\begin{prop} \label{TrivialProp}
Let $E_{t}$ be any eigenbranch. 
Then
\[   \lim_{t \rightarrow 0^+}~  t^2  \cdot E_{t}  \]
exists and is finite. 
\end{prop}

\begin{proof}
Since $\cosh^2(x)\geq 1$, by comparing (\ref{qEllipse})
and  (\ref{qDotEst}) we find that
\[ t \cdot \dot{q}_U(\psi)~ \geq~ -q( \psi ) \]
and by comparing  (\ref{NEllipse}) and  (\ref{NDotEst}) 
\[ t \cdot \dot{N}_U( \psi)~ \leq~ N( \psi). \]
Hence by Corollary \ref{GlobalReduction} we have
\begin{equation} \label{Trivial} 
   \dot{E}~ \geq~  -\left( C~ +~ \frac{2}{t}\right) E    
\end{equation}
for some constant $C$. Lemma \ref{ExpTrick} then allows to conclude 
since in this case $F(t)=Ct^2E_t$ for some constant $C$. 
%
%
%
%
\end{proof}

\subsection{Convergence relative to $t^{2k}$} 

By Proposition \ref{TrivialProp}, the parameter 
$h= t \cdot \sqrt{E}$ in the radial Mathieu equation 
associated to $\psi$ is uniformly bounded in $t$.
This will allow us to prove the following. 

\begin{thm} \label{Bootstrap}
Let $E_{t}$ be any eigenbranch. 
Then there exists $k<1$ such that
\[   \lim_{t \rightarrow 0^+}~  t^{2k}  \cdot E_{t} \]
exists and is finite. 
\end{thm}

\begin{proof}
It suffices to show that  there exists $\kappa< 1$ 
so that 
\begin{equation} \label{BootEstimate}
 t \cdot\dot{N}_U(\psi)~ \leq~ \kappa \cdot 
N(\psi). 
\end{equation}
For then we could argue as in the proof of Proposition 
\ref{TrivialProp} where the `2' that appears in
(\ref{Trivial}) is replaced by $\kappa +1$.  
In particular, the desired $k$ equals $(\kappa + 1)/2$.    

\begin{lem} \label{NonTrivialU}
There exists $M>0$ such that
\[  \int_{U_t} |\psi|^2~ dx d \theta~
 \leq~ M \int_{U_t} |\psi|^2~ \sinh^2 x~ dx d \theta. \]
\end{lem}

Assuming the lemma, we finish the proof of the theorem. 
Let $0 < \kappa <1$ be such that $M = \kappa/ (1- \kappa)$.
Then  
\[
(1- \kappa) \int_{U_t} |\psi|^2 \sin^2\theta \,dxd\theta~
\leq~
(1- \kappa) \int_{U_t} |\psi|^2  \,dxd\theta~
\leq~
 \kappa \int_{U_t} |\psi|^2 \sinh^2 x\,dxd\theta 
\]
and hence 
\[ 
\int_{U_t} |\psi|^2 \sin^2\theta ~ dxd\theta~ %
\leq\, 
 \kappa \int_{U_t} |\psi|^2 \sinh^2 x~ dxd\theta~
  +~ \kappa \int_{U_t} |\psi|^2 \sin^2 \theta~ dxd\theta. 
\]
Estimate (\ref{BootEstimate}) follows then by comparing
(\ref{NEllipse}) and (\ref{NDotEst}). 
\end{proof}

\begin{proof}[Proof of Lemma \ref{NonTrivialU}]
By Lemma \ref{Decomp}, it suffices to prove that
there exists $M$ so that for all $i$ 
\begin{equation}\label{IndividualBound}
\int_{0}^{Y_{t}}~ |u_i|^2~ dx~ 
\leq~  M \int_{0}^{Y_{t}} |u_i|^2~ \sinh^2 x~ dx.
\end{equation}
Let $p(x)= \sinh^2(x)$,
and let $X= \sinh^{-1}(1)$. 
Since $p$ is increasing on $[0, \infty)$,
the infimum of $p(x)$ over $[X, \infty)$
equals $p(X)=1$. In particular, if $t \leq r_0$, then
$Y_t \geq X$, and we obtain 
\begin{equation}\label{Largex}
 \int_{X}^{Y_{t}}|u_i|^2~ p(x)~ dx~ \geq~
    \int_{X}^{Y_{t}} |u_i|^2~ dx. 
\end{equation}

By Proposition \ref{TrivialProp},  $h=h(t)$ is uniformly bounded.  
Choose $i_0 \in {\mathbb Z}^+$ so that $i_0^2 \geq 2 h^2(t)+ 1/2$ for
all sufficiently small $t$.  Then since $b_i(t) \geq i^2$---see
Appendix \ref{MathieuAppendix}---and $\cosh^2(X)=2$, we have that
for all $x \in [0,X]$ and $i \geq i_0$ 
\[   b_i(t)~ - h(t)^2 \cdot \cosh^2(x)~ \geq~ \frac{1}{2}. \]
Since $\psi$ satisfies either Neumann or Dirichlet conditions 
along the slit, we have either $u'_i(0)=0$ or $u_i(0)=0$ for all $i$.
Note also that $p$ is increasing. 

Thus, for $i \geq i_0$, we may apply Proposition \ref{Convex2} 
with $w=u_i^2$ to find that
\[
\int_{0}^X |u_i|^2~ dx~
\leq~ \frac{2}{p(X/2)}
\int_{0}^X |u_i|^2~ p(x)~ dx.
\]
The lemma is thus proved for $i \geq i_0$ large enough.

For $i <i_0$, we note that by 
Proposition \ref{TrivialProp}, $h$ is bounded,
and hence $b_i$ is bounded. Thus, the claim for $i < i_0$ follows
from the fact that the solution of the ordinary differential 
equation (\ref{UODE}) with a fixed boundary condition depends 
continuously on parameters.
\end{proof}

In the next section we make crucial 
use of the following variant of Theorem \ref{Bootstrap}.
\begin{coro} \label{hControl}
Let $h(t)= t \cdot \sqrt{E_t}$. 
There exists $\epsilon_0>0$ so that  
\[   \lim_{t \rightarrow 0}~ t^{-\epsilon_0}h(t)~ =~ 0. 
\]
\end{coro} 


\subsection{Convergence of $E_t$}

\begin{thm} \label{Convergence}
Let $E_{t}$ be a real-analytic eigenbranch.
Then 
\[   \lim_{t \rightarrow 0^+}~   E_{t}  \]
exists and is finite. 
\end{thm}

\begin{proof}
It suffices to show that there exists $\delta>0$ 
so that for all sufficiently small $t$
\begin{equation} \label{conq}
    \dot{q}_U(\psi)~ \geq~
        -~ t^{\delta-1} \cdot q_U( \psi)~ -~ 
         2t^{\delta-1} \cdot E \cdot N_U(\psi),
\end{equation}
and 
\begin{equation} \label{conN} 
   \dot{N}_U(\psi)~ \leq~
        t^{\delta-1} \cdot N_U(\psi).
\end{equation}
For then by Corollary \ref{GlobalReduction}, 
there would exist $C>0$ so that
\begin{equation}  \label{log}
{\dot{E}}~ \geq~ -\left( C~ +~ 4 \cdot t^{\delta-1}\right) E. 
\end{equation}
Since $t^{\delta-1}$ is integrable near $t=0$, the claim
would then follow from Lemma 
\ref{ExpTrick}.

By comparing (\ref{qEllipse}) to (\ref{qDot}) and
(\ref{NEllipse}) to (\ref{NDotEst}), we see
that to prove (\ref{conq}) and (\ref{conN}) it suffices to 
prove that there exists $\delta>0$ so that for sufficiently
small $t$ we have 
\begin{equation} \label{qco}
  \int_{U_t}~ \frac{|\partial_{\theta}
 \psi_t|^2}{\cosh^2 x}~ dxd \theta~
   \leq~ t^{\delta} \left( \int_{U_t}~
  |\partial_{\theta} \psi_t|^2~ dxd \theta~
  +~ 2 t^2 \cdot E \cdot \int_{U_t}
   \psi^2 \sinh^2 x~ dx d\theta \right)
\end{equation}
and
\begin{equation} \label{Nco}
  \int_{U_t}~  |\psi_t|^2~ dxd \theta~
 \leq~ t^{\delta}  \int_{U_t}~ | \psi_t|^2~ \sinh^2 x~ dxd \theta.
\end{equation}

These estimates can be further reduced to estimates 
of Mathieu functions as in \S \ref{ConvexitySection}.  
In order to state these estimates,
we first note that, by Corollary \ref{hControl}, 
there exists $\epsilon>0$ such that  
\begin{equation} \label{Epsilon}
  \lim_{t\rightarrow 0} t^{-2\epsilon} \cdot h(t)~ =~ 0. 
\end{equation}
We fix this $\epsilon$ in what follows. 

\begin{lem} \label{ConAfterSeparate}
Let $u_i$ be a solution to the radial Mathieu equation (\ref{UODE})
with $h=h(t)$ satisfying (\ref{Epsilon}). 
There exists $t_0$ so that if $t<t_0$, then for all $i \geq 0$
\begin{equation} \label{Nco2}
 \int_0^{Y_t} |u_i|^2~ dx~ \leq~  t^{\epsilon/2} 
   \int_{0}^{Y_t}~ |u_i|^2 \sinh^2 x~ dx. 
\end{equation}
and for $i  > 0$
\begin{equation} \label{qco2}
   \int_{0}^{Y_t}~ \frac{|u_i|^2}{\cosh^2 x}~ dx~
   \leq~ t^{\epsilon/2} \int_0^{Y_t}~ |u_i|^2~ dx 
\end{equation}
\end{lem}

Assuming this lemma, we finish the proof of 
Theorem \ref{Convergence}. Using 
(\ref{Decompweight}) we see that (\ref{Nco}) 
follows  immediately from (\ref{Nco2})
with $\delta=\epsilon/2$.

To verify (\ref{qco}) we use (\ref{DecompweightTheta}) 
to find that  
\[  
\int_{U_{t}} \frac{\left|\partial_{\theta}\psi \right|^2}{
    \cosh^2(x)~ } dx d \theta~
\leq~ \sum_{i=0}^{\infty}~ b_i
 \int_{0}^{Y_{t}} \frac{u_i^2(x)}{\cosh^2(x)}~ dx~ 
\] 
In Appendix \ref{MathieuAppendix} we show that 
$b_0 \sim \frac{1}{2} h^2$ for small $h$. 
Hence since $h^2=t^2 E$, we find that for small $t$ 
\begin{eqnarray*}  
 b_0 \int_{0}^{Y_{t}} \frac{u_0^2(x)}{\cosh^2(x)}~ dx~
 &\leq&  h^2 \int _{0}^{Y_{t}} u_0^2(x)~ dx  \\  
 &\leq&  E \cdot t^2 \cdot t^{\epsilon/2}
 \int _{0}^{Y_{t}} u_0^2(x)~ \sinh^2 x~ dx, \\
  &\leq&  t^{\epsilon/2} \cdot  t^2 \cdot E \int_{U_t} |\psi|^2 
     \sinh^2 x~ dx d\theta.
\end{eqnarray*}
Here we have used (\ref{Nco2}). Using  (\ref{DecompweightTheta})
and (\ref{qco2}),   we find that
\begin{eqnarray*}
 \sum_{i=1}^{\infty}~ b_i
 \int_{0}^{Y_{t}} \frac{u_i^2(x)}{\cosh^2(x)}~ dx 
  &\leq &  t^{\delta}~ \sum_{i=1}^{\infty} b_i
    \int_{0}^{Y_{t}} u_i^2(x)~ dx~  \\
 &\leq& t^{\delta}~ 
        \left( \int_{U_t} |\partial_\theta \psi|^2 dxd\theta\,
     +\,     h^2\int_{U_t} |\psi|^2 dxd\theta \right).
\end{eqnarray*}
Thus, since $h^2 =t^2 E$, and using \refeq{qco2} to estimate the last integral of the previous inequality, 
equation (\ref{qco}) follows with $\delta= \epsilon/2$.    
\end{proof}

\subsection{The individual estimates}

The remainder of this section is dedicated to proving Lemma
\ref{ConAfterSeparate}. The behavior of the radial Mathieu 
function $u_i$ for small $h$ is qualitatively different if
$i=0$ than if $i \neq 0$. We prove the cases $i \neq 0$
with Lemma \ref{NonZeroModes} and the case $i=0$ with 
Lemma \ref{ZeroModeForNorm}.  

\subsubsection{Estimates for $i\neq 0$}

\begin{lem}\label{NonZeroModes}
There exists some $t_0$ and some constants $C_1,C_2$ such that for any $t<t_0$ and 
for any $i \neq 0$, we have 
\begin{equation}\label{IndForDer}
\int_{0}^{Y_{t}} \frac{u_i^2(x)}{\cosh^2(x)}~ dx~ 
    \leq~  C_1 \cdot t^{\epsilon}
   \int_{0}^{Y_{t}}  u_i^2(x)~ dx,
\end{equation}
and 
\begin{equation}\label{IndForNorm}
\int_{0}^{Y_{t}} u_i^2(x)~ dx~ \leq~  C_2\cdot t^{\epsilon}
\int_{0}^{Y_{t}}  u_i^2(x)\sinh^2(x)~ dx.
\end{equation}
\end{lem}
\begin{proof}
Define $X_t \in [0, \infty)$ implicitly by  
$\cosh(X_{t})= t^{-\epsilon}$. 
Then we have
\[  \int_{X_{t}}^{Y_{t}} \frac{u_i^2(x)}{\cosh^2(x)}~ dx~ 
\leq~ t^{2\epsilon} \int_{X_{t}}^{Y_{t}} u_i^2(x)~ dx.
\]
Since $t^{2\ep}=o(t^\ep)$, It remains to estimate the integral over $[0, X_t]$.

By (\ref{Epsilon}), we have that $h(t)^2 \cdot \cosh(X_t)^2$
tends to zero as $t$ tends to zero. Since $i \neq 0$, we 
also have $b_i \geq 1$. (See Appendix 
 \ref{MathieuAppendix}). Thus, for sufficiently small $t$,
\[ b_i(h(t))~ - h(t)^2 \cdot \cosh^2(X_t)~ \geq~  \frac{1}{2}. \]
Therefore, we can apply the estimates of \S \ref{ConvexitySection}
with $w=u^2$ and $X=X_t$. By applying Proposition \ref{Convex} 
with $p=\cosh^{-2}$, we find that
\[   \int_{0}^{X_{t}} \frac{u_i^2(x)}{\cosh^2(x)}~ dx~ \leq~   
 \frac{2}{\cosh(X_{t}/2)}
   \int_{0}^{X_{t}}  u_i^2(x)~ dx.
\]
Since $\cosh(X_t/2)=
 \sqrt{(\cosh(X_t)+1)/2}= \sqrt{(t^{-\epsilon}+1)/2}$,
we have
\[   \int_{0}^{X_{t}} \frac{u_i^2(x)}{\cosh^2(x)}~ dx~ \leq~   
 \frac{2 \sqrt{2} \cdot t^{\epsilon}}{1+ t^{\epsilon}}
   \int_{0}^{X_{t}}  u_i^2(x)~ dx.
\]
Estimate  (\ref{IndForDer}) then follows for $C_1=2\sqrt{2}$ and any  $t$ small enough so that 
$t^{2\ep}\leq C_1\cdot t^\ep.$

The proof of the estimate (\ref{IndForNorm}) is similar. 
From the definition of $X_t$ we have
\[  \int_{X_{t}}^{Y_{t}} u_i^2~ dx~ 
\leq~  \frac{2t^{2\epsilon}}{1-t^{2\epsilon}} 
  \int_{X_{t}}^{Y_{t}} u_i^2 \cdot \sinh(x)^2~ dx, 
\] 
and from Corollary \ref{Convex} we find that
\[
 \int_{0}^{X_{t}} u_i^2(x)~ dx~ \leq~ 
\frac{8t^\epsilon}{1-t^\epsilon} 
\int_{0}^{X_{t}} \sinh^{2}(x) \cdot u_i^2(x)~ dx, 
\]
the estimate follows with any $C_2>8$ and $t$ small enough so that 
$\frac{2t^{2\ep}}{1-t^{2\ep}}$ and $\frac{8}{1-t^{\ep}}$ are less than $C_2.$
\end{proof}

Since equation \refeq{IndForDer} implies equation \refeq{qco2}, and
 equation \refeq{IndForNorm} implies equation \refeq{Nco2},
the estimates of  Lemma \ref{ConAfterSeparate} 
are proven for any index $i>0.$ 
It remains to prove \refeq{Nco2} for $i=0.$ 

\subsubsection{The estimate for $i=0$}

Since $b_0(h)$ is real-analytic in $h^2,$ we can write  
\[   b_0(h)~ =~  a(h)\cdot h^2 \]
for some analytic function $a$. It follows from Theorem \ref{Bootstrap}
that $a(h(t))$ is bounded near $t=0$. 
The Mathieu equation for $i=0$ can thus be rewritten as
\begin{equation} \label{ZeroEquation2}
 u_0''(x)~ =~
 h(t)^2 \cdot \left( a(h(t))-\cosh^2(x)\right) \cdot
 u_0(x).
\end{equation}
We will treat as this equation as 
a perturbation of $u_0''=0$. To make
this precise, we set $t_0>0$ and define 
\[ M(X)~ =~ \sup \left\{~  |a(h(t))-\cosh^2(x)|~ 
    |~ 0\leq x \leq X  \ 
{\rm and} \  0 \leq t \leq t_0 \right\}.
\]

\begin{lem}  \label{ZeroEstimate}
Given a solution $u$ to (\ref{ZeroEquation2})
define $R$ by
\[ R(x)~ =~  u(x)~ -~ 
  \left( u(0)~ +~ u'(0)\cdot x \right). \]
Then for all $t \in (0,t_0)$, and for any positive weight $p,$ 
the function $R$ satisfies 
\[
  \int_{0}^X |R(x) |^2~ p(x)~ dx~ 
 \leq~
   h(t)^4 \cdot  M(X)^2 \cdot X^4 \cdot L(X) \cdot 
 \int_{0}^X   |u(x)|^2~ dx
\]
where
\[ L(X)~ =~ \sup_{x \in [0,X]} p(x).\]
\end{lem}
\begin{proof}
By applying the method of variation of constants (or,
equivalently, Duhamel's principle) 
we find that 
\[ R(x)~ =~ h(t)^2 \int_0^x (x-y) 
          \left(a(h(t))-\cosh^2(y) \right)~ u(y)~ dy. \] 
Thus, for $t < t_0$ and $0<x< X$
\[  |R(x) |^2~
     \leq ~ h^4 \cdot M(X)^2 \left( \int_0^x (x-y) 
         \cdot |u(y)|~ dy \right)^2. \]
By applying the Bunyakovsky-Cauchy-Schwarz inequality
we find that
\begin{eqnarray*}
  |R(x) |^2~
     & \leq & 
  h^4 \cdot M^2(X) \left| \int_0^x (x-y)^2 dy \right| 
      \cdot   \left| \int_0^x  u(y)^2~ dy \right| \\
& \leq &
  h^4 \cdot M^2(X) \cdot (|x|^3/3) \cdot 
 \left|\int_0^x  u(y)^2~ dy \right|
\end{eqnarray*}
for all $|x| \leq X$.
Thus,   
\begin{eqnarray*}
  \int_{0}^X |R(x) |^2~ p(x)~ dx 
     & \leq & 
  h^4 \cdot M^2(X) \cdot X^3 \cdot L(X) 
 \int_{0}^X \left|\int_0^x  u(y)^2~ dy \right|~ dx.  \\
\end{eqnarray*}
The desired estimate follows.
\end{proof}

The following completes the proof of 
Lemma \ref{ConAfterSeparate}.
\begin{lem}\label{ZeroModeForNorm}
Let $\epsilon>0$ be as in (\ref{Epsilon}). 
There exists some constant $C$ such that, if we let 
$u=u_0$ be a solution to the radial Mathieu 
equation  (\ref{UODE}) associated to $b_0(t)$ and $h(t)$, and 
if $u$ satisfies either $u(0)=0$ or $u'(0)=0$, 
then for any sufficiently small $t$ we have 
\[ 
\int_{0}^{Y_t} u^2~ \leq~ C \cdot t^{\epsilon}\,
\int_{0}^{Y_t} u^2(x)~ \sinh^2(x)~ dx.
\]
\end{lem}

\begin{proof}
As above, let $X_t \geq 0$ be defined by  
$\cosh(X_{t})= t^{-\epsilon}$.
Let $a=u'(0)$ and $b=u(0)$. Since 
$u(x)^2\leq 2 (ax+b)^2 +2 R(x)^2$, by
applying Lemma \ref{ZeroEstimate} 
with $p(x) \equiv 1$, we find that
\begin{equation} \label{UZeroPrelim}
  \int_{0}^{X_{t}}~ u^2(x)~ dx~ \leq~  
2 \int_{0}^{X_{t}}~ (ax+b)^2~ dx~ +~
 K(t) \int_{0}^{X_{t}}   u^2(x)~ dx
\end{equation}
where 
\[  K(t)~ =~ 2~ M^2(X_t )\cdot h(t)^4  \cdot X_{t}^4. \]
In other words,
\begin{equation} \label{UZeroPrelim2}
  \int_{0}^{X_{t}}~ u^2(x)~ dx~ \leq~  
 \frac{2}{1- K(t)} \int_{0}^{X_{t}}~ (ax+b)^2~ dx,~ 
\end{equation}
provided that $K(t)$ is less than $1.$
As $t$ tends to zero, $X_t \sim \epsilon |\ln t|$ 
tends to infinity, and hence 
$M(X_t) \sim \cosh^2(X_t) \sim t^{-2\epsilon}$.
It follows that 
$K(t) \sim \left(t^{-\epsilon} \cdot h \cdot\epsilon 
|\ln\,h|\right)^4$. Thus, by (\ref{Epsilon}) 
$K(t)$ tends to zero as $t$ tends to zero.
Therefore, for $t$ sufficiently small
\begin{equation}\label{UZeroStep1}
 \int_{0}^{X_{t}}~ u^2(x)~ dx~ 
\leq~ 4 \,\int_{0}^{X_{t}}~ (ax+b)^2~ dx~.
\end{equation}

If $u(0)=0$, then $ax+b=u'(0) \cdot x$, and if 
$u'(0)=0$, then $ax+b=u(0)$. A straightforward 
calculation gives a constant $C$ such that 
for all sufficiently large $X$, we have 
\[ \int_{0}^{X} 1~ dx~  \leq~
 \frac{C \cdot X}{\sinh^2(X)}\int_0^{X}  \sinh^2(x)\,dx. \]
and 
\[ \int_{0}^{X} x^2~ dx~  \leq~
 \frac{ C \cdot X}{\sinh^2(X)}\int_0^{X} x^2 \sinh^2(x)\,dx. \]
Thus, since $\lim_{t \rightarrow 0} X_t= \infty$,   
for small $t$ we have 
\begin{equation} \label{StandardCalc}
 \int_{0}^{X_{t}}~ (ax+ b)^2~ dx~ 
\leq~  \frac{C \cdot X_t}{\sinh^2(X_t)}  
 \int_0^{X} (ax+b)^2 \sinh^2(x)~ dx.
\end{equation}
By applying Lemma \ref{ZeroEstimate} with $p(x)=\sinh^ 2(x)$,
we  obtain 
\[   \int_0^{X_t} (ax+b)^2 \sinh^2(x)\,dx~
 \leq~ 2\int_0^{X_t} u(x)^2 \sinh^2(x)\,dx~
     +~ \tilde{K}(t)  
     \int_0^{X_t} u(x)^2\,dx 
\]
where $\tilde{K}(t)= K(t) \cdot \sinh^2(X_t)$. 
By combining this with (\ref{UZeroStep1}) and (\ref{StandardCalc}) we have
\[
\int_{0}^{X_{t}}~ u^2(x)~ dx~ \leq~ 
\frac{8 C \cdot X_t}{\sinh^2(X_t)}
 \int_0^{X_t} u^2(x) \sinh^2(x)~ dx~ 
+~ 
 4C \cdot X_t \cdot K(t)\int_0^{X_{t}}u^2(x)~ dx. 
\]
Arguing as above, one sees that $X_t \cdot K(t)$ 
tends to zero as $t$ tends to zero. 
It follows that for all sufficiently small $t$, we have
\[
\int_{0}^{X_t} u^2~ dx~ \leq~
  \frac{9C \cdot X_t}{\sinh^2(X_t)}
\int_0^{X_t} u^2(x) \sinh^2(x)~ dx. 
\]
Note that  $C X_t/\sinh^2(X_t) \leq t^{\epsilon}$ 
for sufficiently small $t$.
 
Since 
\[ \inf_{[X_t, Y_t]} \sinh^2(x)~
   =~ \sinh^2(X_t)~ =~ t^{-2\epsilon}-1, \]
we also have
\[
\int_{X_t}^{Y_{t}} u^2(x)~ dx~  
\leq~ \frac{t^{2\epsilon}}{1-t^{2\epsilon}}
    \int_{X_t}^{Y_t} u^2(x)\sinh^2(x) dx, \]
The claim follows.
\end{proof}


\section{Generic simplicity for polygons}

In this section we combine the convergence of analytic
eigenvalue branches with the convergence of eigenfunctions
to generalize our earlier results \cite{HJ1} on spectral 
simplicity of simply-connected polygons to several settings.

\subsection{Slits and  simplicity}

Our results on spectral simplicity 
for polygons depend on the following.  

\begin{prop} \label{SlitSimplicity}
Let $\Sigma_t$ be a slit of length $2t$ 
centered at a point $p$ belonging to the interior 
of a Lipschitz domain $\Omega \subset {\mathbb R}^2$. 
Let $D$ be a measurable subset of $\partial \Omega$
and let $D'_t \subset \Omega_{\Sigma}$ be either $D$ or 
$D\cup \Sigma$. If the spectrum of $q$ on 
$H^1_D(\Omega)$ is simple, then for all 
but countably many $t$, the spectrum of 
$q$ on  $H^1_{D'_t}(\Omega_{\Sigma_t})$ is 
simple. 
\end{prop}

\begin{proof}
By Theorem \ref{Analytic}, the eigenvalues vary
analytically for $t>0$. Hence it suffices to 
show that there does not exist a real-analytic 
eigenvalue branch $E_t$ such that the dimension of the 
associated eigenspace $V_t$ is greater than 1 for 
each $t>0$. 

Suppose that $\psi_t$ and  
$\psi^*_t \in V_t$ denote normalized real-analytic 
eigenfunction branches that are mutually orthogonal 
for each $t$. By Theorem \ref{Convergence}, the corresponding 
eigenvalue branch 
converges to some $E_0$ when $t$ goes to $0$. 
By Theorem \ref{OrderedConvergence}, 
$\psi_t$ and $\psi^*_t$ converge to eigenfunctions 
$\psi$ and $\psi^*$ on $\Omega_{\Sigma}$ with the 
same eigenvalue $E_0$.  Since $\psi_t$ and 
$\psi^*_t$ are normalized and orthogonal for each $t$, 
the limits $\psi$ and $\psi^*$ are orthogonal.  
But this contradicts the assumption that the spectrum of  
$\Omega_{\Sigma}$ is simple. 
\end{proof}

\begin{remk} \label{MultipleSlits}
The proof applies equally well if $\Omega$
is a slit domain with simple spectrum. 
Thus one can iterate the procedure, 
and prove that for all but countably choices
of slit lengths, the domain $\Omega$
slit along finitely many slits has simple
spectrum. Actually, if we take all the slits of the same length 
$t$, the spectrum is also simple for all but countably choices of 
$t$ but to prove this result one has to adapt 
Theorem \ref{Convergence} 
to the case of several slits. This is easily done using the 
localization principle and the estimates we have proved.
\end{remk}

By combining Theorem \ref{SlitSimplicity},
Remark \ref{MultipleSlits},
and the main result of \cite{HJ1}, we obtain the 
simplicity of the spectrum of the generic 
simply connected slit polygon. To be precise, 
let ${\mathcal S}_{n,k}$ be the set of simply connected 
$n$-gons $P$ with $k$ disjoint slits 
$\Sigma_1, \Sigma_2, \cdots, \Sigma_k$. 
Note that the vertices of $P$ and the endpoints 
of each slit $\Sigma_1, \Sigma_2, \cdots, \Sigma_k$, 
determine the slit polygon 
$P_{\Sigma_1 \cup \Sigma_2 \cup \cdots \cup \Sigma_k}$.  Thus, ${\mathcal S}_{n,k}$
may be naturally identified with an open subset 
of ${\mathbb R}^{2n+2k+2k}$. 
In particular, ${\mathcal S}_{n,k}$
inherits a natural affine structure and Borel measure. 
   
\begin{thm} \label{SlitSimple}
If $n \geq 4$ and $k \geq 0$, then almost every slit 
polygon in ${\mathcal S}_{n,k}$
has simple Dirichlet or Neumann spectrum. 
\end{thm}

\begin{proof}
Let $P$ be a simply connected $n$-gon. 
Let $C=\{c_1, \ldots c_k\}$ be a set of $k$ 
distinct points in $P$, and let  
$L= \{\ell_1, \ldots, \ell_k\}$ be 
a set of (not necessarily distinct) lines 
that pass through the origin in ${\mathbb R}^2$. 
Let $A(P,C,L)$ be the set of slit polygons  
$P_{\Sigma_1 \cup \Sigma_2 \cup \cdots \cup \Sigma_k}$
where $\Sigma_i$ is centered at $c_i$ and is
parallel to $\ell_i$. Note that a point 
in $A( P,C, L)$ is determined by the 
lengths of the slits. 
 
The sets $A(P,C, L)$ provide a natural
smooth foliation of ${\mathcal S}_{n,k}$
by $k$-dimensional planar sets. 
By Proposition  \ref{SlitSimplicity} and 
Remark \ref{MultipleSlits}, if $P$ 
has simple spectrum, then almost every
slit polygon in $A(P,C,L)$ has simple 
spectrum.  By the main result of 
\cite{HJ1}, almost every simply connected 
$n$-gon has simple spectrum. The result 
follows by integrating transversely to this foliation using 
Fubini's theorem. 
\end{proof}

\subsection{Polygons with mixed boundary conditions}

Let $P$ be a simply connected polygon,
and let $D$ be a union of a set of edges in 
the boundary, $\partial P$,  of $P$. 
Recall from \S \ref{SectionProblem} 
that the eigenfunctions of the 
Dirichlet energy $q$ on $H^1_D(P)$ satisfy 
Dirichlet conditions on $D$ and Neumann 
conditions on $\partial P \setminus D$.
We say that $(P,D)$ has simple spectrum if and only if
the quadratic form $q$  on $H^1_D(P)$ has
simple spectrum.

Let $v_1, v_2, \ldots, v_n$ be a cyclic ordering 
of the edges of a simply connected polygon $P$. 
Let $e_i$ denote the boundary edge joining the vertex $v_i$ 
to $v_{i+1}$.\footnote{Here $e_n$ joins $v_n$ and $v_1$.} 
Given a subset $\sigma \subset \{1, 2, \ldots, n\}$, let 
$\Poly_{n,\sigma}$ denote the set of pairs 
$(P,D)$ where $P$ is a simply connected
$n$-gon \footnote{By {\em simply connected polygon}, 
we mean 
a compact set whose boundary consists of finitely many
line segments and whose interior is simply connected.} 
 and 
\[ D~ =~  \bigcup_{i \in \sigma}   e_i. \]
The set $\Poly_{n,\sigma}$ can be naturally identified 
with an open subset of ${\mathbb R}^{2n}$ and hence
has a natural Borel measure and affine structure. 
\begin{thm}\label{Mixed}
If $n \geq 4$, then almost every polygon in $\Poly_{n,\sigma}$ has simple spectrum.
\end{thm}

\begin{proof}
By arguing as in \cite{HJ1}, we see that it suffices 
to construct one polygon in $\Poly_{n,\sigma}$ 
that has simple spectrum.

We first construct such a polygon in the case
of alternating boundary conditions, that is, 
we suppose that $i \in \sigma$ if and only if 
$i+1 \notin \sigma$ where as usual $n+1 \equiv 1$.  
Note that $n$ is even in this case.   

Let $Q$ be the rectangle $[0,a] \times [0,1]$,  
and let $D=\{0,1\} \times [0,a]$. That is, 
we consider the eigenvalue problem on $Q$ 
with Dirichlet conditions on the 
vertical edges and Neumann conditions on the 
horizontal edges. The set of functions of the form 
\[   \sin \left( \frac{m_1\pi}{a} \cdot x \right) \cdot
     \cos \left( m_2\pi \cdot y \right),  
\]
where $m_1>0$ and $m_2\geq 0$ are 
integers, is an orthonornal 
basis of eigenfunctions. Thus, the set of eigenvalues 
is $\{ (m_1\pi/a)^2 + (m_2 \pi)^2\}$.
In particular, $(Q,D)$ has simple spectrum if and only if
$a^2 \notin {\mathbb Q}$. 

Let $c_1, c_2, \ldots c_k$ be $k=n/2-2$
distinct points on the horizontal segment 
$[0,a] \times \{1/2\}$. By Theorem \ref{SlitSimplicity}, 
there exist horizontal slits 
$\Sigma_1, \ldots, \Sigma_k$ centered respectively
at $c_i$ such that 
$(Q_{\Sigma_1 \cup \cdots \cup \Sigma_k}, D')$ has 
simple spectrum where 
$D'= D \cup \Sigma_1 \cup \cdots \cup \Sigma_k$. 
See Figure \ref{Trick}.


\begin{figure}[h]
\begin{center}

  \psfrag{(0,1)}{$(0,1)$}
  \psfrag{(0,0)}{$(0,0)$}
  \psfrag{(a,0)}{$(a,0)$}
  \psfrag{(a,0)}{$(a,1)$}
  \psfrag{Q}{$Q$}
  \psfrag{Sigma_1}{$\Sigma_1$}
  \psfrag{Sigma_2}{$\Sigma_2$}

\includegraphics{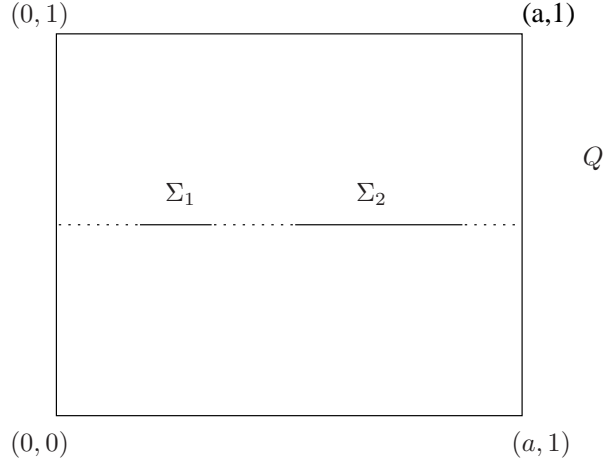}

\end{center}

\caption{\label{Trick} The slit rectangle}

\end{figure}


The slit domain $Q_{\Sigma_1 \cup \cdots \cup \Sigma_k}$
has a reflection symmetry $\tau$ induced by 
$(x,y) \mapsto (x,1-y)$. A standard argument 
shows that restricting functions on
$Q_{\Sigma_1 \cup \cdots \cup \Sigma_k}$ 
to the rectangle $Q'=[0, a] \times [0,1/2]$
defines a bijection between $\tau^*$-invariant 
eigenfunctions of $q$ on 
$H^1(Q_{\Sigma_1 \cup \cdots \cup \Sigma_k}, D')$
and the eigenfunctions of $q$ on $H^1_{D'}(Q')$.  

Therefore, since the spectrum of 
$(Q_{\Sigma_1 \cup \cdots \cup \Sigma_k}, D')$
is simple, the spectrum of $(Q', D')$ is also
simple.  By placing vertices at the endpoints
of the slit, we may regard the rectangle $Q'$
as a polygon with $n=2k+4$ vertices. From this
viewpoint, the boundary conditions given by $D'$
alternate as desired. 

Given a general subset $\sigma \subset \{1, 2, \ldots, n\}$, 
let $F$ be the finite set obtained by identifying 
$i$ and $i+1$ if they either both belong
to $\sigma$ or if they both do not belong to $\sigma$. 
The cyclic ordering of $\{1, 2, \ldots, n\}$ 
induces a cyclic  ordering on $F$, and hence we may 
identify $F$ with $\{1,2, \ldots n'\}$ for some 
$n' \leq n$. Note that $n'$ is either $1$ or even. 
Let $\sigma' \subset \{1,2, \ldots n'\}$ 
denote the set obtained by identifying $i \in \sigma$
with $i+1$ if it also belongs to $\sigma$. 

The set $\sigma' \subset \{1, 2, \ldots, n'\}$
is alternating, and hence, if $n' \geq 4$, there exists an
element $(P',D)$ in $\Poly_{n', \sigma'}$
with simple spectrum. By judiciously adding 
vertices to the boundary edges of the $n'$-gon $P$, 
we obtain an $n$-gon $P'$ such that the set 
$D$ corresponds to $\sigma$. Thus $(P, D)$
is an element of $\Poly_{n, \sigma}$ with simple
spectrum. 

If $n'=2$, then we consider the eigenvalue problem 
on $(Q,D)$ where $Q=[0,a] \times [0,1]$  and $D$ is 
either $\{0\} \times [0,1]$ or its complement. 
In either case a basis of eigenfunctions 
can be constructed by taking products of sines
and cosines. By making an appropriate choice
of $a$, one finds that the spectrum of $(Q,D)$
is simple.  By adding vertices appropriately
to the boundary edges of $Q$, we obtain 
$(Q',D) \in \Poly_{n, \sigma}$ with simple spectrum.

The case $n'=1$ is the main result of \cite{HJ1} and follows from the 
same construction. 
\end{proof}

\subsection{Mutliply connected polygons} 
We first make precise the definition of multiply connected 
polygon. Let $P_0, P_1, \ldots, P_k$ 
be a finite collection of simply connected 
polygons. Assume that 
\begin{enumerate}
\item $P_i$ is contained in the interior of $P_0$ 
       for all $i>0$, and 
\item $P_i \cap P_j = \emptyset$ for all $i,j \geq 0,~i \neq j$,  
\end{enumerate}
The {\em multiply connected polygon $P$ determined
by $P_0, P_1\ldots, P_k$} is obtained by removing 
the interiors of the polygons 
$P_1, \ldots P_k$ from the polygon $P_0$.  
In other words, $P$ is the closure of 
\[  P_0 \setminus  \left( \bigcup_{i=1}^k  P_i \right).
\]

Let $\vec{n}=(n_0, n_1, \cdots, n_k)$ 
denote a vector of integers with $n_i \geq 3$ 
for each $i$.  Let $\Poly(\vec{n})$ 
denote the set of all collections of 
polygons $P_0,  P_1, \ldots, P_k$ satisfying
(1) and (2) above and such that for each $i$, the polygon 
$P_i$ has $n_i$ (ordered) vertices 
$v_{i,1}, \ldots, v_{i,n_i}$. Since the 
(ordered) vertices determine the polygon,
$\Poly(\vec{n})$ is naturally in bijective correspondence 
with an open subset of ${\mathbb R}^d$ 
where $d= 2n_0 + \cdots + 2n_k$.
In particular,  $\Poly(\vec{n})$ inherits
an affine structure and a Borel measure.

\begin{prop} \label{Connected}
The space $\Poly(\vec{n})$ is path connected. 
\end{prop}

\begin{proof}
In \cite{HJ1}, we proved that the space $\Poly_n$
of simply connected $n$-gons is connected
using a construction which we will call `deleting a vertex'. 
In particular,  given a simply connected 
$n$-gon $P$, we constructed a linear path $t \mapsto P(t)$ 
of $n$-gons in $\Poly_n$  with $P(0)=P$ and such that $P(1)$ has 
three consecutive vertices that belong to the same
boundary edge. See Figure \ref{FigDeleting}.
We can regard the polygon $P(1)$ 
as an element of $\Poly_{n-1}$. Thus, since the space of 
triangles, $\Poly_3$, is connected, $\Poly_n$ is connected
for $n \geq 3$ by induction.


\begin{figure}[h]
\begin{center}

  \psfrag{$v_2$}{$v_2$}
  \psfrag{$v_1$}{$v_1$}
  \psfrag{$v_n$}{$v_n$} 
  \psfrag{$v_t$}{$v_t$} 
  \psfrag{$m$}{$m$} 

\includegraphics{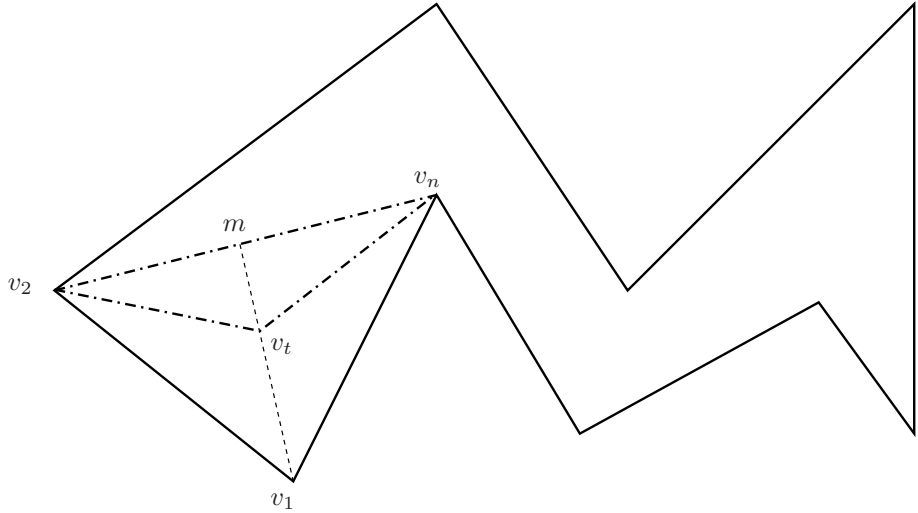}

\end{center}

\caption{\label{FigDeleting}Deleting a vertex}

\end{figure}


Suppose now that $P \in \Poly(\vec{n})$ is defined 
by the simply connected polygons $P_0, P_1, $\\
$\ldots, P_k$
where $n_i > 3$ for some  $i >0$.  Let 
$t \mapsto P_i(t)$ denote the path from $P_i=P_i(0)$
to a polygon $P_i(1)$ having three consecutive vertices
on the same edge. Let $t \mapsto P(t)$ denote the path of polygons 
determined by $P_1, \ldots, P_{i-1}, P_i(t), P_{i+1}, \cdots, P_k$.
The polygon $P(1)$ may be regarded as 
an element of $\Poly(n_0, \ldots, n_{i-1}, 
                   n_i-1, n_{i+1}, \ldots, n_k)$. 
 
Therefore, by induction on the vector $\vec{n}$, 
we find that it suffices to prove that the space
$\Poly(n,3,3,3, \ldots, 3)$ is connected. 
By rescaling the `interior' triangles if necessary, 
we may delete vertices of the polygon $P_0$
to obtain a path to an element of 
$\Poly(3,3,3,3, \ldots, 3)$.
An elementary argument then gives that
$\Poly(3,3,3,3, \ldots, 3)$ is connected.
\end{proof}

\begin{thm}
If $n_0 \geq 4$, then almost every polygon in  
$\Poly(\vec{n})$ has simple spectrum. 
\end{thm}

\begin{proof}
Since $\Poly(\vec{n})$ is connected, by arguing  
as in \cite{HJ1},  we see that it suffices 
to construct one polygon in $\Poly (\vec{n})$ 
that has simple spectrum.

Let $P$ belong to $\Poly(4, 3, 3, 3, \ldots, 3)$. 
By judiciously adding vertices to the boundary edges of $P$, 
we may regard $P$ as an element of 
$\Poly(n_0, n_1, n_2,n_3, \ldots, n_k)$
where $n_0\geq 4$ and $n_i \geq 3$ for $i > 0$.  
Thus, it will suffice to prove that there
exists some $P \in \Poly(4, 3, 3,3, \ldots, 3)$
such that $P$ has simple spectrum. 

By Theorem  \ref{SlitSimple} there 
exists a quadrilateral $Q$ with $k$ slits
$\Sigma_1, \Sigma_2, \ldots \Sigma_k$ so that 
$Q_{\Sigma_1 \cup\Sigma_2 \cup \cdots \cup \Sigma_k}$
has simple spectrum. For each $i$, choose 
a point $p_i$ so that the convex hulls, $P_i$,
of $\{p_i \} \bigcup \Sigma_i$
satisfy conditions (1) and (2) above. Let 
$m_i$ be the midpoint of $\Sigma_i$ and define 
the path $x_i(t)= t p_i + (1-t) m_i$. 
Define $P_i(t)$ to be the convex hull of 
$\{x_i \} \bigcup \Sigma_i$. Let $P(t)$ be the multiply
connected polygon defined by 
$Q, P_1(t), P_2(t), \ldots, P_k(t)$. 
See Figure \ref{Open}.
Note that $P(t)$ is an element of $\Poly(4,3,3,\ldots,3)$
for $t>0$ and that $P(0)$ corresponds to  
$(Q, \Sigma_1, \Sigma_2, \ldots, \Sigma_n)$.  


\begin{figure}[h]
\begin{center}

  \psfrag{p_i}{$p_i$}
  \psfrag{e_i}{$e_i$}
  \psfrag{m_i}{$m_i$}
  \psfrag{Q_t}{$Q_t$}  
  \psfrag{P}{$P$}
  \psfrag{Qhalf}{$Q_{\frac{1}{2}}$}

\includegraphics{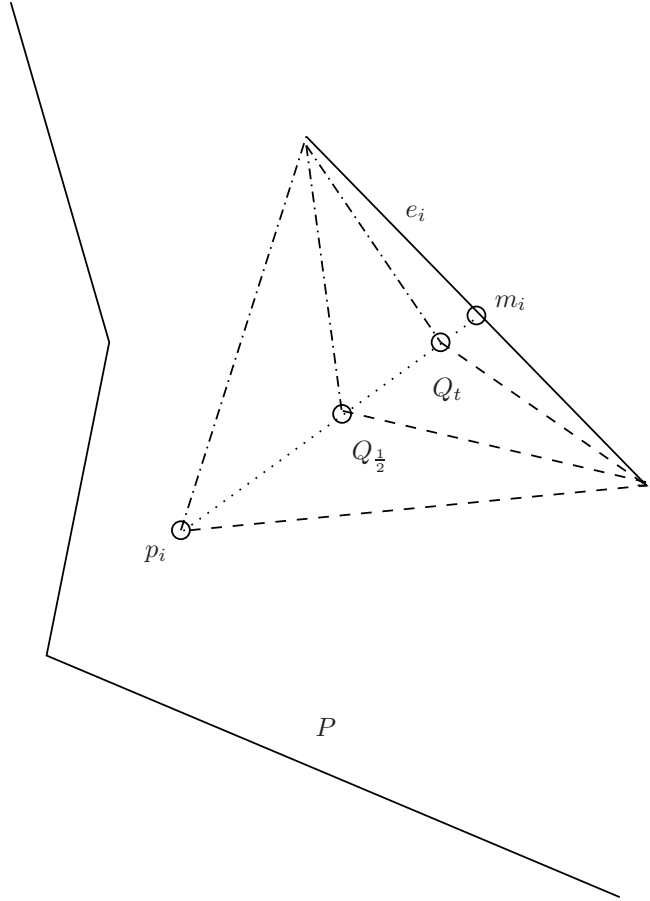}

\end{center}

\caption{\label{Open}Perturbing a slit into a triangle}

\end{figure}


We define a linear path $t \mapsto f_t$ of 
piecewise linear homeomorphisms mapping
$P(1/2)$ onto  $P(t)$ for each $t \in [0,1/2]$. 
In particular, let $Q_i(t)$ be the quadrilateral 
$P_i(1) \setminus P_i(t)$, and let $f^i_t$
be the `obvious' piecewise linear homeomorphism that maps 
$Q_i(1/2)$ onto $Q_i(t)$. Define 
$f_t: P(1/2) \rightarrow P(t)$ by
\[  f_t(x)~ =~ \left\{ \begin{array}{cl}
                x, & \mbox{ if } x \notin \bigcup_i Q_i(1/2) \\
                f_t^i(x), & \mbox{ if } x \in Q_i(1/2).
                 \end{array} \right. 
\]

 By Lemma 2.1 in \cite{HJ1}, the eigenvalues of $P(t)$
vary real-analytically in $t$. Since the spectrum 
of $P(0)$ is simple, there exists $t_0>0$
such that the spectrum of $P(t_0)$ is simple. 
The claim is proven.   
\end{proof}


\appendix

\section{Mathieu functions}

\label{MathieuAppendix}

For the convenience of the reader we provide 
basic facts about Mathieu functions that are used in 
the present paper. For additional information, 
we refer the reader to \cite{Mathieu} and 
\cite{MorseFeshbach}. Our approach is based on 
analytic perturbation theory. See Example VII.3.4 in \cite{Kato}.

Suppose that  $\Delta \psi= E \cdot \psi$. A straightforward 
computation shows that  
\begin{equation} \label{PreSeparate}
 - \left(\partial_z^2~ + \partial_{\theta}^2\right)~ (\psi \circ F_{t})~ =~
   t^2 \cdot E \cdot
   \left(~ \cosh(z)^2~ -~ \cos^2(\theta) \right) \cdot (\psi \circ F_{t}).
\end{equation}
Let $h= t \sqrt{E}$. 
The method of separation of variables leads one to consider
the operator 
\[  A_h~ =~ -\frac{d^2}{d\theta^2}~  +~ h^2 \cos^2(\theta)~ \]
acting self-adjointly on $L^2( {\mathbb R} / (2 \pi {\mathbb Z}), d\theta)$.

\begin{prop}[Angular Mathieu functions] \label{Angular}
For each $i \in {\mathbb Z}$, there exist unique
real-analytic paths $v_i: {\mathbb R} \rightarrow 
L^2( {\mathbb R} / (2 \pi {\mathbb Z}), d\theta)$ and  
$b_i: {\mathbb R} \rightarrow {\mathbb R}$ so that
for each $h \in {\mathbb R}$
\begin{enumerate} 
\renewcommand{\labelenumi}{(\alph{enumi})}
\item $v_i(h)$ is an eigenfunction  
      for $A_h$ with eigenvalue $b_i(h)$:
\[    A_h \left(v_i(h) \right)~ = b_{i}(h) \cdot v_i(h). \]

\vspace{.3cm}

\item $v_i(h)$ has unit norm in   
      $L^2( {\mathbb R} / (2 \pi {\mathbb Z}), d\theta)$, 

\vspace{.3cm}

\item the span of $\{ v_i(h)~ |~  i=0,1,2 \ldots \}$ is
      dense in $ L^2( {\mathbb R} / (2 \pi {\mathbb Z}), d\theta)$, 

\vspace{.3cm}

\item For $i$ odd
\[ v_i(h)(\theta)~ =~ \pi^{-\frac{1}{2}} \cdot \sin(i\theta)~ +~ O(h^2),\]
and for $i$ even
\[ v_i(h)(\theta)~ =~ \pi^{-\frac{1}{2}} \cdot \cos(i\theta)~ +~ O(h^2).\] 
\item For $i$ odd
    \[   b_i(h)~ =~  i^2~ +~ 2h^2~ +~ O(h^4), \]
   and for $i$ even 
     \[   b_i(h)~ =~ i^2~ +~ \frac{1}{2}h^2~ +~ O(h^4) \]
\item For any $i$ and any $h,$ we have $b_i(h)\geq b_i(0).$ 
\end{enumerate}  
\end{prop}

\begin{remk}
The functions 
$v_{i}(h)$ are known classically as angular Mathieu functions.
and we will use indifferently the 
notations $v_i(h)$ and $v_{i,h}.$
See for example \cite{MorseFeshbach}.
\end{remk}

\begin{proof}
The path $h \mapsto A_h$ is an analytic family 
of compactly resolved operators. Hence the existence
of paths $v_i$ and $b_i$ satisfying (a), (b), and (c)
follows from Theorem VII.3.9 on page 392 in \cite{Kato}. 

The $k^2$-eigenspace of $A_0= \partial_{\theta}^2$ 
is spanned by ${\mathcal B}_i=\{\sin(i\theta),
\cos(i\theta)\}$. Let $P_i$ be the
orthogonal projection onto this eigenspace.  
Let $\ddot{A}$ denote the 
second derivative of $h \mapsto A_h$ evaluated at $h=0$.
One computes the matrix of $P_i \ddot{A} P_i$ with respect to the 
orthonormal basis ${\mathcal B}_i$ to be
\[ \left( \begin{array}{cc}  1/2  & 0 \\ 0 & 2 \end{array}
\right). \]
Uniqueness of the paths as well as property (d) then follow 
from analytic perturbation theory \cite{Kato}. 
The derivative of $b_i$ is given by 
\[ \frac{d}{dh}b_i(h)\,=\,2h \int_0^{2\pi} 
\cos^2(\theta)v_{i,h}(\theta)^2\,d\theta.
\]
Thus $b_i$ is increasing for positive 
$h$ and decreasing for negative $h$ showing that 
$h=0$ is a global minimum.
\end{proof}

\begin{coro} \label{MathieuRepresentation}
Let $\psi$ be an eigenvector of $\Delta$ with Dirichlet 
boundary condition on the slit (resp. Neumann) then 
\[ \psi(z, \theta)~ =~ \sum_i u_{i,h}(z) \cdot v_{i,h}(\theta). \]
where $$v_{i,h}= v_{i}(h)$$
are as above, and $u=u_{i,h}$ satisfies 
the ordinary differential equation
\begin{equation}  \label{UODE2}  
-~ u''(z)~ 
    +~ \left(~  b_{i}(h)~ -~ h^2 \cosh^2(z) \right) 
     \cdot u(z)~ =~ 0,
\end{equation}
with Dirichlet boundary condition at $z=0$ (resp. Neumann).
\end{coro}

\begin{proof}
Each eigenfunction $\psi$ is smooth
and hence by Proposition \ref{Angular}  
\[ \psi(z, \theta)~ =~ \sum_i u_{i,h}(z) \cdot v_{i,h}(\theta). \]
where $h^2= t^2 E$, and  
\[  u_{i,h}(r)~ =~ \int_0^{2 \pi} \psi(z, \theta) 
            \cdot v_{i,h}(\theta)~ d \theta.
\]
One obtains (\ref{UODE}) by using (\ref{PreSeparate}) 
and integrating by parts.
\end{proof}



\noindent
 \\
Laboratoire de Math\'ematiques Jean Leray\\
UMR CNRS 6629-Universit\'e de Nantes\\
2 rue de la Houssini\`ere\\
BP 92 208\\
F-44 322 Nantes Cedex 3\\ 
{\tt Luc.Hillairet@math.univ-nantes.fr}\\

\noindent
Department of Mathematics, \\
Indiana University, Bloomington, IN, 47401 \\
{\tt cjudge@indiana.edu} 

\end{document}